\documentclass[a4paper,12pt]{amsart}
\usepackage[T1]{fontenc}    
\usepackage[utf8]{inputenc}
\usepackage{enumerate}
\usepackage{amssymb}
\usepackage{fullpage}
\usepackage[colorlinks=true, linkcolor=red, citecolor=blue]{hyperref}

\def\RR{\mathbb R}
\def\CC{\mathbb C}
\def\ZZ{\mathbb Z}
\def\TT{\mathbb T}

\def\AA{\mathcal A}
\def\BB{\mathcal B}
\def\HH{\mathcal H}
\def\GG{\mathcal G}

\newcommand{\set}[1]{\left\lbrace #1\right\rbrace}
\providecommand{\abs}[1]{\left\lvert#1\right\rvert}
\providecommand{\norm}[1]{\left\lVert#1\right\rVert}
\newcommand{\remove}[1]{ }
\newcommand{\qtq}[1]{\quad\text{#1}\quad}

\DeclareMathOperator{\dist}{dist}

\newtheorem{theorem}{Theorem}[section]
\newtheorem{proposition}[theorem]{Proposition}
\newtheorem{lemma}[theorem]{Lemma}
\newtheorem{corollary}[theorem]{Corollary}

\theoremstyle{definition}
\newtheorem*{definition}{Definition}

\theoremstyle{remark}
\newtheorem{remark}[theorem]{Remark}
\newtheorem{remarks}[theorem]{Remarks}\remove{}

\numberwithin{equation}{section}

\begin{document}
\title[Pointwise control]{Pointwise control of the linearized Gear--Grimshaw system}
\author[Capistrano--Filho]{Roberto A. Capistrano--Filho}
\address{Departamento de Matemática,
Universidade Federal de Pernambuco,
S/N Cidade Universitária,
50740-545, Recife (PE)
Brazil}
\email{capistranofilho@dmat.ufpe.br}
\author[Komornik]{Vilmos Komornik}
\address{College of Mathematics and Computational Science, Shenzhen University, Shenzhen 518060, People’s Republic of China and Département de mathématique,  Université de Strasbourg, 7 rue René Descartes, 67084 Strasbourg Cedex, France}
\email{komornik@math.unistra.fr}
\author[Pazoto]{Ademir F. Pazoto}
\address{Instituto de Matemática,
Universidade Federal do Rio de Janeiro,
C.P. 68530 - Cidade Universitária - Ilha do Fundão,
21945-970 Rio de Janeiro (RJ),
Brazil}
\email{ademir@im.ufrj.br}
\date{Version 2019-09-22}

\begin{abstract}\mbox{}
In this paper we consider the problem of controlling pointwise, by means of a time dependent Dirac measure supported by a given point, a coupled system of two Korteweg--de Vries equations on the unit circle. More precisely, by means of spectral analysis and Fourier expansion we prove,  under general assumptions on the physical parameters of the system, a pointwise observability inequality which leads to the pointwise controllability by using two control functions. 
In addition, with a uniqueness property proved for the linearized system without control, we are also able to show pointwise controllability when only one control function acts internally. 
In both cases we can find, under some assumptions on the coefficients of the system, the sharp time of the controllability.
\end{abstract}
\maketitle

\section{Introduction}\label{s1}

Wave phenomena occur in many branches of mathematical physics and due to the wide practical applications it has become one of the most important scientific research areas.
During the past several decades, many scientists developed mathematical models to explain the wave behavior. The Korteweg-de Vries equation (KdV)
\begin{equation*}
u_t + u_{xxx}+u u_x=0
\end{equation*}
was first proposed as a model for propagation of unidirectional, one-dimensional, small-amplitude long waves of water in a channel. 
A few of the many other applications include internal gravity waves in a stratified fluid, waves in a rotating atmosphere, ion-acoustic waves in a plasma, among others.
Starting in the latter half of the 1960s, the mathematical theory for such nonlinear dispersive wave equations came to the fore
as a major topic within nonlinear analysis. 
Since then physicists and mathematicians were led to derive sets of equations to describe the
dynamics of the waves in some specific physical regimes and much effort has been expended on various aspects of the initial and boundary value problems.
For instance, since the first coupled KdV system was proposed by Hirota and Satsuma \cite{HirSat1981,HirSat1982}, it has been studied amply  and some important coupled KdV models have been
derived. 
In particular, general coupled KdV models were applied in different fields, such as in shallow stratified liquid:
\begin{equation}\label{11}
\begin{cases}
u_t + u_{xxx} +  a_3v_{xxx} + uu_x + a_1vv_x + a_2(uv)_x =  0, \\
b_1v_t + rv_x + v_{xxx} + b_2a_3u_{xxx} +vv_x + b_2a_2uu_x + b_2a_1(uv)_x = 0,
\end{cases}
\end{equation}
where $u=u(x,t)$ and $v=v(x,t)$ are real-valued functions of the real variables $x$ and $t$, and
$a_1$, $a_2$, $a_3$, $b_1$, $b_2$ and $r$ are real constants with $b_1>0$ and $b_2>0$.
System \eqref{11} was proposed by Gear and Grimshaw \cite{GeaGri1984} as a model to describe strong interactions of two long internal gravity waves
in a stratified fluid, where the two waves are assumed to correspond to different modes of the linearized equations of motion. 
It has the structure
of a pair of KdV equations with linear coupling terms and has been object of intensive research in recent years. 
We refer to
\cite{BonPonSauTom1992} for an extensive discussion on the physical relevance of the system in its full structure.

\subsection{Setting of the problem} 
In this paper we are mainly concerned with the study of the pointwise controllability of the linearized Gear--Grimshaw system posed on the unit circle $\TT $:
\begin{equation}\label{12}
\begin{cases}
u_t+u_{xxx}+av_{xxx}=f(t)\delta_{x_0}\qtq{in}\RR\times\TT,\\
cv_t+rv_x+v_{xxx}+du_{xxx}=g(t)\delta_{x_0}\qtq{in}\RR\times\TT,\\
u(0)=u_0\qtq{and} v(0)=v_0\qtq{in}\TT,\\
\end{cases}
\end{equation}
where $a,c,d,r$ are given positive constants, $\delta_{x_0}$ denotes the Dirac delta function centered in a given point $x_0\in\TT$ and $f,g$ are the control functions.

More precisely, the purpose is to see whether one can force the solutions of those systems to have certain desired properties by choosing appropriate control inputs. 
Consideration will be given to the following fundamental problem that arises in control theory, as proposed by Haraux in \cite{Har1990}:

\medskip

\noindent\textbf{Pointwise control problem:} Given $x_0\in\TT$, $T > 0$ and  $(u_0,v_0)$, $(u_T,v_T)$ in $L^2(\TT)\times L^2(\TT)$, can we find appropriate $f(t)$ and $g(t)$ in a certain space such that the corresponding solution $(u,v)$ of \eqref{12} satisfies 
\begin{equation*}
u(T)=u_T \qtq{and}v(T)=v_T?
\end{equation*}

If we can always find control inputs to drive the system from any given initial state $(u_0,v_0)$ to any given terminal state $(u_T,v_T)$, then the system is said to be pointwise controllable.

\subsection{State of the art}  
As far as we know, the internal controllability problem for the system \eqref{11} remains open.
By contrast, the study of the boundary controllability properties is considerably more developed. 
Indeed, the first result
was obtained in \cite{micuortega2000}, when the model is posed on a periodic domain and $r = 0$. 
In this case, a diagonalization of the main terms allows us to decouple the corresponding linear system into two scalar KdV equations and to use the previous results available in the literature. 
Later on, Micu \emph{et al.} \cite{micuortegapazoto2009} proved the local exact boundary controllability property for the nonlinear system, posed on a bounded interval. 
Their result was improved by Cerpa and Pazoto \cite{cerpapazoto2011} and by Capistrano--Filho \emph{et al.} \cite{capisgallegopazoto2016}. 
By considering a different set of boundary conditions, the same boundary control problem was also addressed by the authors in  \cite{capisgallegopazoto2017}. 
We note that, the results mentioned above were first obtained for the corresponding linearized systems
by applying the Hilbert Uniqueness Method (HUM) due to J.-L. Lions \cite{Lions1988}, combined with some ideas introduced by Rosier in \cite{rosier1997}. 
In this case the problem is reduced to prove the so-called \emph{``observability inequality''} for the corresponding adjoint system. 
The controllability result was extended to the full system by means of a fixed point argument.

The internal stabilization problem has also been addressed; see, for instance, \cite{CaKoPa2014,Dav3,PazSou} and the references therein. 
Although controllability and stabilization  are closely related, one may expect that some of the available results will have some counterparts in the context of the control problem, but this issue is open. Particularly, when the model is posed on a periodic domain, Capistrano--Filho \emph{et al.} \cite{CaKoPa2014}, designed a time-varying feedback law and established the exponential stability of the solutions in Sobolev spaces of any positive integral order by using a Lyapunov approach. 
This extends an earlier theorem of Dávila \cite{Dav3} also obtained in $H^s(\TT)$, for $s\leq 2$.
The proof follows the ideas introduced in \cite{KomRusZha1991} for the scalar KdV equation by using the infinite family of conservation laws for this equation. 
These conservation laws lead to the construction of suitable Lyapunov functions that give the exponential decay of the solutions.
In \cite{CaKoPa2014} the Lyapunov approach was possible thanks to the results established by Dávila and Chavez \cite{DavCha2006}. 
They proved that, under suitable conditions on the coefficients of the system, the system also has an infinite family of conservation laws.

\subsection{Main results} 
As we mentioned before, no results are available in the literature on the internal controllability of the Gear-Grimshaw system. 
In this work we use spectral analysis and Fourier series to prove some results of pointwise controllability for the system \eqref{12}.

Fourier series are considered to be very useful in linear control theory (see, e.g. \cite{Rus1978} and its references). 
In particular, a classical generalization of Parseval’s equality, given by Ingham \cite{Ing1936}, and its many recent variants are very efficient in solving many control problems where other methods do not seem to apply. 
An outline of this theory is presented in \cite{BaiKomLor111,KomLor2005,Lor2005}.

We also prove some new results concerning the use of harmonic analysis in the framework of dispersive systems. 
In this spirit, we derive the controllability of the linearized Gear-Grimshaw system posed on the unit circle $\mathbb{T}$. 
As it was pointed out earlier by Haraux and Jaffard \cite{Har1990,HarJaf991,Har1994} by studying some other systems, for non-periodic  boundary conditions the controllability properties  may heavily depend on the location of the observation or control point.

One of the main result provides a sharp positive answer for the controllability issue mentioned in the beginning of this introduction.

\begin{theorem}\label{t11}
For almost all quadruples $(a,c,d,r)\in(0,\infty)^4$ the following property holds.
For any fixed
\begin{equation*}
x_0\in\TT,\quad T>0\qtq{and}
(u_0,v_0), (u_T,v_T)\in H:=L^2(\TT)\times L^2(\TT)
\end{equation*}
there exist control functions $f,g\in L^2_{\text{loc}}(\RR)$ such that the solution of \eqref{12} satisfies the final conditions
\begin{equation*}
u(T)=u_T\qtq{and} v(T)=v_T.
\end{equation*}
\end{theorem}

We prove this theorem by applying the Hilbert Uniqueness Method (HUM) due to J.-L. Lions \cite{Lions1988} (see also Dolecki and Russell \cite{DolRus1977}) that reduces the controllability property to the observability of the homogeneous dual problem
\begin{equation}\label{13}
\begin{cases}
u_t+u_{xxx}+av_{xxx}=0
\qtq{in}\RR\times\TT,\\
cv_t+rv_x+v_{xxx}+du_{xxx}=0\qtq{in}\RR\times\TT,\\
u(0)=u_0\qtq{and} v(0)=v_0\qtq{in}\TT.
\end{cases}
\end{equation}
More precisely, Theorem \ref{t11} will be obtained as a consequence of

\begin{theorem}\label{t12}
For almost all quadruples $(a,c,d,r)\in(0,\infty)^4$ the following properties hold.

\begin{enumerate}[\upshape (i)]
\item Given any $(u_0,v_0)\in H$, the system \eqref{13} has a unique solution $(u,v)\in C_b(\RR,H)$, and the linear map
\begin{equation*}
(u_0,v_0)\mapsto (u,v)
\end{equation*}
is continuous from $H$ into $C_b(\RR,H)$.

\item The \emph{energy} of the solution, defined by the formula
\begin{equation*}
E(t):=\norm{(u,v)(t)}_H^2 
=\int_{\TT}\abs{u(t,x)}^2+\frac{ac}{d}\abs{v(t,x)}^2\ dx,
\end{equation*}
does not depend on $t\in\RR$.

\item For every solution and $x_0\in\TT$ the functions $u(\cdot,x_0)$ and $v(\cdot,x_0)$ are well defined in $L^2_{loc}(\RR)$.
\item For every non-degenerate bounded interval $I$ there exist two positive constants $\alpha, \beta$  such that
\begin{equation*}
\alpha E\le
\int_I\abs{u(t,x_0)}^2+\abs{v(t,x_0)}^2\ dt
\le \beta E
\end{equation*}
for all solutions of \eqref{13} and for all $x_0\in\TT$.
\end{enumerate}
\end{theorem}

By applying a general method \cite{Kom83}, analogous to HUM, Theorem \ref{t12} will also imply the pointwise exponential stabilizability of \eqref{12}:

\begin{theorem}\label{t13}
For almost all quadruples $(a,c,d,r)\in(0,\infty)^4$ the following property holds.
For any fixed $x_0\in\TT$ and $\omega>0$ there exist two continuous linear functionals
\begin{equation*}
F:H\to \RR\qtq{and}G:H\to \RR
\end{equation*}
such that the following properties hold.
\begin{enumerate}[\upshape (i)]
\item Given any $(u_0,v_0)\in H$, the system
\begin{equation*}
\begin{cases}
u_t+u_{xxx}+av_{xxx}=F(u,v)\delta_{x_0}
\qtq{in}\RR\times\TT,\\
cv_t+rv_x+v_{xxx}+du_{xxx}=G(u,v)\delta_{x_0}
\qtq{in}\RR\times\TT,\\
u(0)=u_0\qtq{and} v(0)=v_0\qtq{in}\TT
\end{cases}
\end{equation*}
has a unique solution $(u,v)\in C_b(\RR,H)$, and the linear map
\begin{equation*}
(u_0,v_0)\mapsto (u,v)
\end{equation*}
is continuous from $H$ into $C_b(\RR,H)$.
\item There exists a constant $M>0$ such that
\begin{equation*}
\norm{(u(t),v(t))}_H\le Me^{-\omega t}\norm{(u_0,v_0)}_H
\end{equation*}
for all solutions and for all $t\ge 0$.
\end{enumerate}
\end{theorem}

Another relevant result of this work is a uniqueness theorem when only one function, $u(\cdot,x_0)$ or  $v(\cdot,x_0)$, is observed.

\begin{theorem}\label{t14}
For almost all quadruples $(a,c,d,r)\in(0,\infty)^4$ the following property holds.

Fix $x_0\in\TT$ and a non-degenerate interval $I$ arbitrarily, and consider a solution of \eqref{13}.

\begin{enumerate}[\upshape (i)]
\item If $u(t,x_0)=0$ for all $t\in I$, then  $u=0$ and $v$ is an arbitrary constant function.
\item If $v(t,x_0)=0$ for all $t\in I$, then  $v=0$ and $u$ is an arbitrary constant function.
\end{enumerate}
\end{theorem}

\begin{remarks}\mbox{}
\begin{enumerate}[\upshape (i)]
\item Theorems \ref{t11}, \ref{t12},  \ref{t13} and \ref{t14} will be proved in sharper forms, by making the assumptions of $a,c,d,r$ more explicit, and considering also some 
cases where the results hold only under a sharp condition $\abs{I}>T_0$ or $T>T_0$ with some explicitly given $T_0$, where $\abs{I}$ denotes the length of the interval $I$.
    
It is known that Fourier series and, in particular, Ingham type inequalities, are very efficient in solving many control problems when the frequencies satisfy 
a uniform gap condition. 
However, in some control problems only some weakened gap conditions are satisfied; see, for instance, \cite{BaiKomLor111,CaZu1996,JaTuZu1997}. 
Under such assumptions we may still get some {\em weakened observability theorems}.
Theorem \ref{t14} is obtained by applying a general result of this type.
Similarly to Theorem \ref{t12}, Theorem \ref{t14} implies some exact controllability and exponential stabilizability results by acting in only one of the equations.

\item The above results remain valid for the scalar KdV equation.
We will not present the proofs because they are similar (and simpler) than the proofs given in this paper.

\item If we require more regularity on the initial data, say $(u_0,v_0)\in H^2(\TT)\times H^2(\TT)$, then the results obtained for the linear system allow us to prove the local controllability of the nonlinear system by means of a fixed point argument.
The proof is similar to that of \cite[Theorem 2.2]{capistranogallego2018},
and hence it will be omitted.
\end{enumerate}
\end{remarks}

The plan of the present article is the following.
In Section \ref{s2} we present some known and new vectorial Ingham type theorems which will form the basis of the proofs of our  observability and uniqueness theorems. Then Theorems \ref{t12}, \ref{t11}, \ref{t13} and \ref{t14} will be proved (in strengthened forms) in Sections \ref{s3}, \ref{s4}, \ref{s5} and \ref{s6}, respectively.

\section{A review of Ingham type theorems}\label{s2}

Given a family $\Omega=(\omega_k)_{k\in K}:=\set{\omega_k\ :\ k\in K}$ of real numbers, we consider functions of the form
\begin{equation}\label{21}
\sum_{k\in K}c_ke^{i\omega_kt}
\end{equation}
with square summable complex coefficients $(c_k)_{k\in K}:=\set{c_k\ :\ k\in K}$, and we investigate the relationship between the quantities
\begin{equation*}
\int_I\abs{\sum_{k\in K}c_ke^{i\omega_kt}}^2\ dt
\qtq{and}
\sum_{k\in K}\abs{c_k}^2
\end{equation*}
where $I$ is some given bounded interval.

We recall the definition of the \emph{upper density} $D^+=D^+(\Omega)$ of $\Omega$.
For each $\ell>0$ we denote by $n^+(\ell)$ the largest number of exponents $\omega_k$ that we may find in an interval of length $\ell$, and then we set
\begin{equation*}
D^+:=\inf_{\ell>0}\frac{n^+(\ell)}{\ell}\in[0,\infty].
\end{equation*}
It can be shown (see, e.g., \cite[p. 57]{BaiKomLor111} or \cite[p. 174]{KomLor2005}) that
\begin{equation*}
D^+=\lim_{\ell\to\infty}\frac{n^+(\ell)}{\ell}.
\end{equation*}
It follows from the definition that $D^+$ is subadditive:
\begin{equation}\label{22}
D^+(\Omega_1\cup\Omega_2)\le D^+(\Omega_1)+D^+(\Omega_2)
\end{equation}
for any families $\Omega_1$ and $\Omega_2$.
If $\Omega$ is \emph{uniformly separated}, i.e., if 
\begin{equation*}
\gamma=\gamma(\Omega)=\inf\set{\abs{\omega_k-\omega_n}\ :\ k\ne n}>0,
\end{equation*}
then $D^+\le 1/\gamma$, and hence $D^+<\infty$.

Henceforth, unless stated otherwise, by a \emph{complex family}  we mean a  \emph{square summable family of complex numbers}.

First we recall a classical theorem of Ingham and Beurling \cite{Ing1936,Beu}.

\begin{theorem}\label{t21}
Assume that the family $\Omega=(\omega_k)_{k\in K}$ is uniformly separated.

\begin{enumerate}[\upshape (i)]
\item The sum \eqref{21} is a well defined function in $L^2_{\text{loc}}(\RR)$ for all complex families $(c_k)_{k\in K}$.
\item For each bounded interval $I$ there exists a positive constant $\alpha(I)$ such that
\begin{equation*}
\int_I\abs{\sum_{k\in K}c_ke^{i\omega_kt}}^2\ dt\le\alpha(I)\sum_{k\in K}\abs{c_k}^2
\end{equation*}
for all complex families $(c_k)_{k\in K}$. 
\item For each bounded interval $I$ of length $>2\pi D^+$ there exists a positive constant $\beta(I)$ such that
\begin{equation*}
\beta(I)\sum_{k\in K}\abs{c_k}^2\le
\int_I\abs{\sum_{k\in K}c_ke^{i\omega_kt}}^2\ dt
\end{equation*}
for all complex families $(c_k)_{k\in K}$. 
\end{enumerate}
\end{theorem}
\noindent In fact, the constants $\alpha(I)$ and $\beta(I)$ depend only on the length of $I$.
The inequalities in (ii) and (iii) are called \emph{direct} and \emph{inverse} inequalities, respectively.

\begin{remark}\label{r22}
If  $A:X\to\RR$ and  $B:X\to\RR$ are two nonnegative functions, defined on some set $X$, then we will write, following Vinogradov, 
$A\ll B$ if there exists a positive constant $\theta$ such that $A(x)\le\theta B(x)$ for all  $x\in X$, and  $A\asymp B$ if $A\ll B$ and $B\ll A$.
For example, in the estimates (ii) and (iii) above the integral and the sum may be considered as real-valued functions defined on the vector space of the complex families $(c_k)_{k\in K}$, and we may write the inequalities in the form
\begin{equation*}
\int_I\abs{\sum_{k\in K}c_ke^{i\omega_kt}}^2\ dt\ll\sum_{k\in K}\abs{c_k}^2
\qtq{and}
\sum_{k\in K}\abs{c_k}^2\ll
\int_I\abs{\sum_{k\in K}c_ke^{i\omega_kt}}^2\ dt,
\end{equation*}
respectively, if we do not want to indicate explicitly the constants $\alpha(I)$ and $\beta(I)$.
\end{remark}

We will need an extension of Theorem \ref{t21} to the more general where $\Omega$ has a finite upper density, but it is not necessarily uniformly separated.
If $D^+(\Omega)<\infty$, then we may enumerate the elements of $\Omega$ into an increasing sequence $(\omega_k)$, where $k$ runs over a (finite or infinite) subset of $\ZZ$, formed by consecutive integers.
Assume for simplicity that $k$ runs over all integers.
Given an arbitrary bounded interval $I$ of length $\abs{I}$, by \cite[Proposition 1.4]{BaiKomLor111} or \cite[Proposition 9.3]{KomLor2005} there exists a positive integer $M$ such that $\Omega$ satisfies the \emph{weakened gap condition}
\begin{equation}\label{23}
\gamma_M=\gamma_M(\Omega):=\inf_k\frac{\omega_{k+M}-\omega_k}{M}>0,
\end{equation}
and 
\begin{equation}\label{24}
2\pi D^+\le \frac{2\pi}{\gamma_M}<\abs{I}.
\end{equation}

Fix $0<\varepsilon\le \gamma_M$ arbitrarily.
For each maximal chain $\omega_k,\ldots,\omega_n$ satisfying
\begin{equation*}
\omega_{j+1}-\omega_j<\varepsilon\qtq{for} j=k+1,\ldots, n
\end{equation*}
(note that it has at most $M$ elements), we introduce the  divided differences $e_k(t)$, \ldots, $e_n(t)$ by Newton's formulas
\begin{equation*}
e_k(t):=e^{i\omega_kt},\quad e_{k+1}(t):=\frac{e^{i\omega_{k+1}t}-e^{i\omega_kt}}{\omega_{k+1}-\omega_k},\ldots
\end{equation*}
(see \cite{BaiKomLor111} or \cite{KomLor2005} for more details), and we rewrite the usual exponential sums in the form
\begin{equation}\label{25}
\sum_{k\in \ZZ} c_ke^{i\omega_kt}=\sum_{k\in \ZZ} b_ke_k(t).
\end{equation}
There exists always a positive constant $\theta$ such that
\begin{equation*}
\sum_{k\in \ZZ}\abs{b_k}^2
\le\theta\sum_{k\in \ZZ}\abs{c_k}^2
\end{equation*}
for all complex families $(c_k)_{k\in K}$. 
On the other hand, there exists a positive constant $\eta$ such that

\begin{equation*}
\sum_{k\in \ZZ}\abs{c_k}^2
\le\eta\sum_{k\in \ZZ}\abs{b_k}^2
\end{equation*}
for all complex families $(c_k)_{k\in K}$ if and only if $\Omega$ is uniformly separated.

\begin{theorem}\label{t23}
Assume that $D^+(\Omega)<\infty$.
\begin{enumerate}[\upshape (i)]
\item The sum \eqref{21} is a well defined function in $L^2_{\text{loc}}(\RR)$ for every complex family $(c_k)_{k\in K}$. 
\item For each bounded interval $I$ there exists a positive constant $\alpha(I)$ such that
\begin{equation*}
\int_I\abs{\sum_{k\in K}c_ke^{i\omega_kt}}^2\ dt\le\alpha(I)\sum_{k\in K}\abs{c_k}^2
\end{equation*}
for all complex families $(c_k)_{k\in K}$. 
\item Fix a bounded interval $I$ of length $\abs{I}>2\pi D^+$, choose $M$ satisfying \eqref{23} and \eqref{24}, then choose an arbitrary $\varepsilon\in(0,\gamma_M]$, and introduce the functions $e_k(t)$ according to \eqref{25}.
Then there exist two positive constants $\alpha(I)$ and $\beta(I)$ such that
\begin{equation*}
\beta(I)\sum_{k\in \ZZ}\abs{b_k}^2
\le\int_I\Bigl\vert\sum_{k\in \ZZ} c_ke^{i\omega_kt}\Bigr\vert^2\ dt
\le\alpha(I)\sum_{k\in \ZZ}\abs{b_k}^2
\end{equation*}
for all complex families $(c_k)_{k\in K}$. 
\end{enumerate}
\end{theorem}
\noindent As before, the constants $\alpha(I)$ and $\beta(I)$ depend only on the length of $I$.

\begin{remarks}\label{r24}\mbox{}
\begin{enumerate}[\upshape (i)]
\item Mehrenberger \cite{Meh2005} proved that $2\pi D^+$ is the critical length for the validity of (iii).
\item If the sequence $(\omega_k)$ has a uniform gap $\gamma>0$, then choosing $\varepsilon\in(0,\gamma]$ every maximal chain is a singleton, so that $b_k=c_k$ for all $k$.
In this special case Theorem \ref{t23} reduces to Theorem \ref{t21}.
\end{enumerate}
\end{remarks}

We will also need a vectorial extension of Theorem \ref{t21} (i), (ii).

\begin{theorem}\label{t25}
Let $\Omega=(\omega_k)_{k\in K}$ be a uniformly separated family of real numbers, and  $(Z_k)_{k\in K}$ a uniformly bounded family of vectors in some Hilbert space $H$, so that
\begin{equation}\label{26}
\Delta:=\sup_{k\in K}\norm{Z_k}<\infty.
\end{equation}

\begin{enumerate}[\upshape (i)]
\item The sum
\begin{equation*}
\sum_{k\in K}c_ke^{i\omega_kt}Z_k
\end{equation*}
is a well defined element of $L^2_{\text{loc}}(\RR,H)$
for every complex family $(c_k)_{k\in K}$.
\item For each bounded interval $I$ there exists a positive constant $\alpha(I)$, independent of $\Delta$, such that
\begin{equation*}
\int_I\Bigl\lVert\sum_{k\in K}c_ke^{i\omega_kt}Z_k\Bigr\rVert^2\ dt
\le\alpha(I) \Delta^2\sum_{k\in K}\abs{c_k}^2
\end{equation*}
for all complex families $(c_k)_{k\in K}$.
\end{enumerate}
\end{theorem}

\begin{proof}
This is an easy adaptation of the usual proof of Theorem \ref{t21} (ii); see, e.g., the proof of \cite[Theorem 6.1, p. 90]{KomLor2005}.
\end{proof}

Finally, we will also need a vectorial generalization of a powerful estimation method of Haraux \cite{Har1989}.
Let us consider a finite number of families of real numbers:
\begin{equation}\label{27}
\Omega_j=(\omega_{j,k})_{k\in K_j},
\quad j=1,\ldots,J,
\end{equation}
and corresponding nonzero vectors $Z_{j,k}$ in some Hilbert space $H$.
The following proposition is a  special case of \cite[Theorem 6.2]{KomLor2005}:

\begin{theorem}\label{t26}
Assume that the families \eqref{27} are uniformly separated. 
Furthermore, assume that there exist finite subsets $F_j\subset K_j$, a bounded interval $I_0$, and two positive constants $\alpha'$, $\beta'$ such that
\begin{equation}\label{28}
\beta'\sum_{j=1}^J\sum_{k\in K_j\setminus F_j}\abs{c_{j,k}}^2
\le\int_{I_0}\abs{\sum_{j=1}^J\sum_{k\in K_j\setminus F_j}c_{j,k}e^{i\omega_{j,k}t}Z_{j,k}}^2\ dt
\le\alpha'\sum_{j=1}^J\sum_{k\in K_j\setminus F_j}\abs{c_{j,k}}^2
\end{equation}
for all complex families
\begin{equation*}
(c_{j,k})_{k\in K_j\setminus F_j},\quad j=1,\ldots,J.
\end{equation*}
Finally, assume that for each exceptional index $(j,k)$ such that $k\in F_j$, the
exponent $\omega_{j,k}$ has a positive distance from the sets
\begin{equation*}
\set{\omega_{j,n}\ :\ n\in K_j,\ n\ne k}
\end{equation*}
and
\begin{equation*}
\set{\omega_{\ell,n}\ :\ n\in K_{\ell}}\qtq{for all}\ell\ne j.
\end{equation*}
Then for each bounded interval $I$ of length $>\abs{I_0}$ there exist two positive constants $\alpha(I)$ and $\beta(I)$ such that
\begin{equation*}
\beta(I)\sum_{j=1}^J\sum_{k\in K_j}\abs{c_{j,k}}^2
\le\int_I\abs{\sum_{j=1}^J\sum_{k\in K_j}c_{j,k}e^{i\omega_{j,k}t}Z_{j,k}}^2\ dt
\le\alpha(I)\sum_{j=1}^J\sum_{k\in K_j}\abs{c_{j,k}}^2
\end{equation*}
holds for all complex complex families
\begin{equation*}
(c_{j,k})_{k\in K_j},\quad j=1,\ldots,J.
\end{equation*}
\end{theorem}

\begin{remark}\label{r27}
The assumption \eqref{28} is stronger than the assumption  \eqref{26} of Theorem \ref{t25}.
Indeed, applying \eqref{28} to one-element sums we obtain that
\begin{equation*}
\beta'\le\abs{I_0}\cdot\norm{Z_{j,k}}^2\le \alpha'
\end{equation*}
for all $j$ and $k\in K_j\setminus F_j$.
Since $F_1\cup\cdots\cup F_J$ is a finite set and the vectors $Z_{j,k}$ are different from zero, we may choose two other positive constants $\beta''\le\beta'$ and  $\alpha''\ge\alpha'$ such that
\begin{equation*}
\beta''\le\abs{I_0}\cdot\norm{Z_{j,k}}^2\le \alpha''
\end{equation*}
for all $j$ and $k\in K_j$.
Hence \eqref{26} is satisfied with $\Delta=\sqrt{\alpha''/\abs{I_0}}$.
\end{remark}

\section{Pointwise observability}\label{s3}

Given four positive constants $a,c,d,r$, we consider the following system of linear partial differential equations with $2\pi$-periodic boundary conditions:

\begin{equation*}
\begin{cases}
u_t+u_{xxx}+av_{xxx}=0\qtq{in}\RR\times(0,2\pi),\\
cv_t+rv_x+v_{xxx}+du_{xxx}=0\qtq{in}\RR\times(0,2\pi),\\
\frac{\partial^j u}{\partial x^j}(t,0)=\frac{\partial^j u}{\partial x^j}(t,2\pi)\qtq{for}t\in\RR,\quad j=0,1,2,\\
\frac{\partial^j v}{\partial x^j}(t,0)=\frac{\partial^j v}{\partial x^j}(t,2\pi)\qtq{for}t\in\RR,\quad j=0,1,2,\\
u(0,x)=u_0(x)\qtq{for}x\in (0,2\pi),\\
v(0,x)=v_0(x)\qtq{for}x\in (0,2\pi).
\end{cases}
\end{equation*}

It will be more convenient to write $u(t)(x):=u(t,x)$, and  to work on the unit circle $\TT $ without boundary conditions, i.e., to rewrite our system in the form

\begin{equation}\label{31}
\begin{cases}
u_t+u_{xxx}+av_{xxx}=0\qtq{in}\RR\times\TT,\\
cv_t+rv_x+v_{xxx}+du_{xxx}=0\qtq{in}\RR\times\TT,\\
u(0)=u_0\qtq{and}v(0)=v_0.
\end{cases}
\end{equation}
Let us write \eqref{31} in the abstract form
\begin{equation*}
Z'+AZ=0,\quad Z(0)=Z_0
\end{equation*}
in the Hilbert space
\begin{equation*}
H:=L^2(\TT )\times L^2(\TT )
\end{equation*}
with
\begin{equation*}
Z=
\begin{pmatrix}
u\\ v
\end{pmatrix}
,\quad
Z_0=
\begin{pmatrix}
u_0\\ v_0
\end{pmatrix}
\end{equation*}
and the linear operator
\begin{equation*}
A=\frac{1}{c}
\begin{pmatrix}
cD^3 & acD^3 \\
dD^3 & rD+D^3
\end{pmatrix}
,\quad D(A):=H^3(\TT)\times H^3(\TT),
\end{equation*}
where $Z'$ is the time derivative and $D$ is the spatial derivative.  The well posedness of \eqref{31} in most cases follows from the following result.

\begin{proposition}\label{p31}
If $ad\ne 1$, then $A$ is an anti-adjoint operator in $H$ for the Euclidean norm given by
\begin{equation*}
\norm{(z_1,z_2)}^2:=\int_{\TT }\abs{z_1}^2+\frac{ac}{d}\abs{z_2}^2\ dx,
\end{equation*}
and hence it generates a group of isometries in $H$.
\end{proposition}

\begin{proof}
It  is clear that $D(A)$ is dense in $H$.
We have to prove that $A^*=-A$.
First we show that $-A\subset A^*$, i.e., $(u,v)\in D(A^*)$ and  $A^*(u,v)=-A(u,v)$ for all $(u,v)\in D(A)$.
Indeed, for any $(\varphi ,\psi),(u,v)\in D(A)$, we obtain by  integrating by parts the equality
\begin{align*}
((\varphi,\psi) ,A(u,v))_H
&=-\int_{\TT }u (\varphi_{xxx}+a\psi_{xxx})dx-\frac{a}{d}\int_{\TT } v(r\psi_x + \psi_{xxx}+d\varphi_{xxx})dx \\
&= -(A(\varphi,\psi),(u,v))_H.
\end{align*}

It remains to show that  $D(A^*)=D(-A)$, i.e., that each $(\varphi,\psi) \in D(A^*)$ belongs to $D(A)$.

Pick any $(\varphi,\psi) \in D(A^*)$.
Then there exists a constant $C>0$ satisfying for all $(u,v)\in D(A)$ the inequality
\begin{equation*}
\abs{((\varphi,\psi) , A(u,v))_H}
\le C \norm{(u,v)}_H,
\end{equation*}
or equivalently
\begin{equation*}
\abs{\int_{\TT }  \varphi (u_{xxx}+av_{xxx})+\frac{a}{d}\psi(rv_x +v_{xxx}+du_{xxx})dx}
\le C \left(\int_{\TT} [u^2+\frac{ac}{d}v^2 ]dx\right)^\frac{1}{2}.
\end{equation*}

Choosing $v=0$ and $u\in C^\infty(\TT)$ hence we infer from that the distributional derivative $\varphi_{xxx}+a\psi_{xxx}$ belongs to $ L^2(\TT)$.
Similarly, choosing $u=0$ and $v\in C^\infty (\TT)$ we obtain that the distributional derivative $d\varphi_{xxx}+r\psi_x+\psi_{xxx}$ belongs to $ L^2(\TT)$.
Combining the two relations we obtain that
$(1-ad)\psi_{xxx}+r\psi_x\in L^2(\TT)$, and in case $ad\ne 1$ this implies that $\psi \in H^3(\TT)$.
Combining this property with the relation $\varphi_{xxx}+a\psi_{xxx}\in L^2(\TT)$ we conclude that $\varphi_{xxx}\in L^2(\TT)$, and therefore $\varphi \in H^3(\TT)$.
\end{proof}

In order to establish the well posedness of \eqref{31} in the general case we determine the eigenvalues and eigenfunctions of the  operator $A$.
There exists an orthogonal basis
\begin{equation}\label{32}
e^{-ikx}Z_k^{\pm},\quad k\in\ZZ
\end{equation}
of $H$, consisting of eigenfunctions of $A$.
Indeed, for each fixed $k$, 
$e^{-ikx}(u_k,v_k)$
is an eigenvector of $A$ with the eigenvalue $i\omega_k$ if and only if
\begin{equation*}
\begin{cases}
(\omega_k-k^3)u_k-ak^3v_k=0,\\
-dk^3u_k+(c\omega_k+rk-k^3)v_k=0.
\end{cases}
\end{equation*}
There exist non-trivial solutions if and only if
\begin{equation*}
\begin{vmatrix}
\omega_k-k^3 & -ak^3\\ -dk^3 & c\omega_k+rk-k^3
\end{vmatrix}
=0,
\end{equation*}
or equivalently if
\begin{equation*}
c\omega_k^2+(rk-(c+1)k^3)\omega_k+(1-ad)k^6-rk^4=0.
\end{equation*}
Hence we have two possible exponents, given by the formula
\begin{equation}\label{33}
2c\omega_k^{\pm}=(c+1)k^3-rk\\
\pm k\sqrt{4acdk^4+\left[(c-1)k^2+r\right]^2}.
\end{equation}

If $k\ne 0$, then $\omega_k^-\ne\omega_k^+$, and two corresponding non-zero eigenvectors are given by the formula
\begin{equation}\label{34}
Z_k^{\pm}
:=2ck^{-3}(ak^3,\omega_k^{\pm}-k^3)
=\left(2ac,1-c-rk^{-2}\pm \sqrt{4acd+\left[c-1+rk^{-2}\right]^2}\right).
\end{equation}

If $k=0$, then both eigenvalues are equal to zero, and two linearly independent eigenvectors are given for example by the
formula
\begin{equation}\label{35}
Z_0^{\pm}:=\left(2ac,\pm \sqrt{4acd}\right).
\end{equation}

\begin{lemma}\label{l32}
The functions \eqref{32} given by \eqref{33}--\eqref{35} form an orthogonal basis in $H$, and
\begin{equation*}
A(e^{-ikx}Z_k^{\pm})=i\omega_k^{\pm}(e^{-ikx}Z_k^{\pm})
\end{equation*}
for all $k\in\ZZ$.

Furthermore, we have
\begin{equation*}
Z_k^{\pm}\to Z^{\pm}:=\left(2ac,1-c\pm \sqrt{4acd+(c-1)^2}\right)\qtq{as}k\to\pm\infty,
\end{equation*}
and $\norm{Z_k^{\pm}}\asymp 1$.
\end{lemma}

According to Remark \ref{r22}, the notation $\norm{Z_k^{\pm}}\asymp 1$ means that there exist two positive constants $\theta$ and $\eta$ such that
\begin{equation*}
\theta \le \norm{Z_k^+}\le\eta
\qtq{and}
\theta \le \norm{Z_k^-}\le\eta
\end{equation*}
for all $k\in\ZZ$.

\begin{proof}
The orthogonality follows from the orthogonality of the functions $e^{-ikx}$ in $L^2(\TT )$ and from the orthogonality relations
\begin{equation*}
Z_k^+\cdot Z_k^-=0\qtq{for all}k\in\ZZ.
\end{equation*}
The latter equalities may be checked directly: we have
\begin{equation*}
Z_0^+\cdot Z_0^-=4a^2c^2+\frac{ac}{d}(-4acd)=0,
\end{equation*}
and
\begin{equation*}
Z_k^+\cdot Z_k^-=4a^2c^2+\frac{ac}{d}\left[\left[1-c-rk^{-2}\right]^2-4acd-\left[c-1+rk^{-2}\right]^2\right]=0
\end{equation*}
if $k\ne 0$.

The limit relations readily follow from \eqref{34} because $k^{-2}\to 0$.
Since $Z^{\pm}\ne 0$, they imply the property $\norm{Z_k^{\pm}}\asymp 1$.
\end{proof}

Next we establish the well posedness of \eqref{31}.
Let us denote by $C_b(\RR,H)$  the Banach space of bounded continuous functions $\RR\to H$ for the uniform norm.

\begin{theorem}\label{t33}\mbox{}

\begin{enumerate}[\upshape (i)]
\item Given any $(u_0,v_0)\in H$, the system \eqref{31} has a unique solution $(u,v)\in C_b(\RR,H)$, and the linear map
\begin{equation*}
(u_0,v_0)\mapsto (u,v)
\end{equation*}
is continuous from $H$ into $C_b(\RR,H)$.

\item The \emph{energy} of the solution, defined by the formula
\begin{equation*}
E(t):=\norm{(u,v)(t)}_H^2
=\int_{\TT}\abs{u(t,x)}^2+\frac{ac}{d}\abs{v(t,x)}^2\ dx,
\end{equation*}
does not depend on $t\in\RR$.

\item The solution is given by the explicit formula
\begin{equation}\label{36}
\begin{pmatrix}
u\\ v
\end{pmatrix}
(t,x)=\sum_{k\in\ZZ}\left(c_k^+e^{i\omega_k^+t}Z_k^++c_k^-e^{i\omega_k^-t}Z_k^-\right)e^{-ikx}
\end{equation}
with suitable square summable complex coefficients $c_k^{\pm}$ satisfying the equality
\begin{equation}\label{37}
\sum_{k\in\ZZ}\left(\abs{c_k^+}^2\cdot\norm{Z_k^+}^2+\abs{c_k^-}^2\cdot\norm{Z_k^-}^2\right)
=2\pi E,
\end{equation}
and hence the relation
\begin{equation}\label{38}
\sum_{k\in\ZZ}\left(\abs{c_k^+}^2+\abs{c_k^-}^2\right)
\asymp E.
\end{equation}
\end{enumerate}
\end{theorem}

\begin{proof}
If $ad\ne 1$, then the theorem follows from  Proposition \ref{p31} and Lemma \ref{l32}.

The following alternative proof works even if $ad=1$.
The functions $(2\pi)^{-1}e^{-ikx}$, $k\in\ZZ$ form an orthonormal basis in $L^2(\TT)$.
Furthermore, by Lemma \ref{l32} the non-zero vectors $Z_k^{\pm}$ are orthogonal in $\CC^2$, and $\norm{Z_k^{\pm}}\asymp 1$.
Hence the functions $e_k^{\pm}(x):=e^{-ikx}Z_k^{\pm}$, $k\in\ZZ$ form an orthogonal basis in $H$, and $\norm{e_k^{\pm}}\asymp 1$.

By a standard method, the theorem will be proved if we show that for any given square summable sequences $(c_k^{\pm})$ the series
\begin{equation*}
U(t,x):=\sum_{k\in\ZZ}\left(c_k^+e^{i\omega_k^+t}Z_k^++c_k^-e^{i\omega_k^-t}Z_k^-\right)e^{-ikx}
\end{equation*}
is uniformly convergent in $C_b(\RR,H)$, and $\norm{U(t)}$ is independent of $t$.
Since $C_b(\RR,H)$ is a Banach space, it suffices to check the uniform Cauchy criterium.
Setting
\begin{equation*}
U_n(t,x):=\sum_{k=-n}^n\left(c_k^+e^{i\omega_k^+t}Z_k^++c_k^-e^{i\omega_k^-t}Z_k^-\right)e^{-ikx},\quad n=1,2,\ldots,
\end{equation*}
the following equality holds for all $n>m>0$
 and $t\in\RR$:
\begin{equation*}
U_n(t,x)-U_m(t,x)
=\sum_{m<\abs{k}\le n}\left(c_k^+e^{i\omega_k^+t}Z_k^++c_k^-e^{i\omega_k^-t}Z_k^-\right)e^{-ikx}.
\end{equation*}
Using the above mentioned orthogonality properties and the relation $\norm{Z_k^{\pm}}\asymp 1$  hence we infer that
\begin{align*}
\norm{U_n(t)-U_m(t)}_H^2
&=2\pi\sum_{m<\abs{k}\le n}\left(\norm{c_k^+e^{i\omega_k^+t}Z_k^+}_{\CC^2}^2+\norm{c_k^-e^{i\omega_k^-t}Z_k^-}_{\CC^2}^2\right)\\
&=2\pi\sum_{m<\abs{k}\le n}\left(\norm{c_k^+Z_k^+}_{\CC^2}^2+\norm{c_k^-Z_k^-}_{\CC^2}^2\right)\\
&\asymp \sum_{m<\abs{k}\le n}\left(\abs{c_k^+}^2+\abs{c_k^-}^2\right).
\end{align*}
The Cauchy property follows by observing that the last expression is independent of $t$, and that it converges to zero as $n>m\to\infty$ by the convergence of the numerical series
\begin{equation*}
\sum_{k\in\ZZ}\left(\abs{c_k^+}^2+\abs{c_k^-}^2\right).
\end{equation*}
The above proof also yields (by taking $m=-1$) the equalities \eqref{36}--\eqref{38}.
\end{proof}

Now we turn to the question of observability.
We need some additional information on the eigenvalues:

\begin{lemma}\label{l34}
We have
\begin{equation*}
\lim_{k\to\pm\infty}(\omega_{k+1}^+-\omega_k^+)=\infty.
\end{equation*}
\end{lemma}

\begin{proof}
Since $\omega_{-k}^+=-\omega_k^+$, it suffices to consider the case $k\to\infty$.
Since
\begin{equation*}
2ck^{-3}\omega_k^+
=c+1+rk^{-2}+\sqrt{4acd+[c-1+rk^{-2}]^2},
\end{equation*}
setting
\begin{equation*}
A:=\frac{c+1+\sqrt{4acd+(c-1)^2}}{2c}\end{equation*}
for brevity, we have
\begin{equation*}
\omega_k^+=Ak^3+O(k)\qtq{as}k\to\infty.
\end{equation*}
Hence
\begin{equation*}
\omega_{k+1}^+-\omega_k^+
=A\left[(k+1)^3-k^3\right]+O(k)
=3Ak^2+O(k)\qtq{as}k\to\infty.
\end{equation*}
Since $A>0$, the lemma follows.
\end{proof}

\begin{lemma}\label{l35}
We have
\begin{equation*}
\lim_{k\to\pm\infty}(\omega_{k+1}^--\omega_k^-)=
\begin{cases}
\infty&\text{if $ad<1$,}\\
-\infty&\text{if $ad>1$,}\\
\frac{-r}{c+1}&\text{if $ad=1$.}
\end{cases} 
\end{equation*}
\end{lemma}

\begin{proof}
Since $\omega_{-k}^-=-\omega_k^-$, it suffices to consider the case $k\to\infty$.
We have
\begin{equation}\label{39}
\begin{split}
2ck^{-3}\omega_k^-
&=c+1-rk^{-2}-\sqrt{4acd+[c-1+rk^{-2}]^2}\\
&=\frac{[c+1-rk^{-2}]^2-4acd-[c-1+rk^{-2}]^2}{c+1-rk^{-2}+\sqrt{4acd+[c-1+rk^{-2}]^2}}\\
&=\frac{2c[2-2rk^{-2}]-4acd}{c+1-rk^{-2}+\sqrt{4acd+[c-1+rk^{-2}]^2}}\\
&=\frac{4c(1-ad)-4crk^{-2}}{c+1-rk^{-2}+\sqrt{4acd+[c-1+rk^{-2}]^2}}.
\end{split}
\end{equation}

If $ad\ne 1$, then the last expression converges to the non-zero number
\begin{equation*}
B:=\frac{4c(1-ad)}{c+1+\sqrt{4acd+(c-1)^2}}
\end{equation*}
as $k\to\pm\infty$.
(It may be shown by a direct computation, although we do not need this in the present proof, that the denominators in \eqref{39} are non-zero for every $k$.)
Hence
\begin{equation*}
\omega_k^-=\frac{B}{2c}k^3+O(k),
\end{equation*}
and therefore
\begin{equation*}
\omega_{k+1}^--\omega_k^-
=\frac{B}{2c}\left[(k+1)^3-k^3\right]+O(k)
=\frac{3B}{2c}k^2+O(k),\quad k\to\pm\infty.
\end{equation*}
This implies the first two cases of the lemma because $B$ and $1-ad$ have the same sign.

If $ad=1$, then \eqref{39} implies the relation
\begin{equation*}
\omega_k^-=\frac{-r}{c+1}k+O(k^{-1})
\end{equation*}
as $k\to\pm\infty$.
Hence
\begin{equation*}
\omega_{k+1}^--\omega_k^-
\to \frac{-r}{c+1},\quad k\to\pm\infty.\qedhere
\end{equation*}
\end{proof}

\begin{corollary}\label{c36}
We have 
\begin{equation*}
D^+(\set{\omega_k^+})=0,
\qtq{and}
D^+(\set{\omega_k^-})=
\begin{cases}
0&\text{if $ad\ne 1$,}\\
\frac{c+1}{r}&\text{if $ad=1$.}
\end{cases} 
\end{equation*}
\end{corollary}

\begin{proof}
By using the definition of the upper density, this follows from Lemmas \ref{l34} and \ref{l35}.
\end{proof}

The rest of this section is devoted to the proof of the following

\begin{theorem}\label{t37}
Assume that
\begin{equation}\label{310}
\omega_k^+\ne\omega_n^+
\qtq{and}
\omega_k^-\ne\omega_n^-
\qtq{whenever}
k\ne n.
\end{equation}
Fix $x_0\in\TT$ arbitrarily, and consider the solutions of \eqref{31}.

\begin{enumerate}[\upshape (i)]
\item The direct inequality
\begin{equation*}
\int_I\abs{u(t,x_0)}^2+\abs{v(t,x_0)}^2\ dt\ll E
\end{equation*}
holds for all non-degenerate bounded intervals $I$.
\item If $ad\ne 1$, then the inverse inequality
\begin{equation}\label{311}
E\ll\int_I\abs{u(t,x_0)}^2+\abs{v(t,x_0)}^2\ dt\end{equation}
also holds for all non-degenerate bounded intervals $I$.
\item If $ad=1$, then  \eqref{311} holds  for all bounded intervals $I$ of length $\abs{I}>2\pi (c+1)/r$, and it fails if $\abs{I}<2\pi (c+1)/r$.
\end{enumerate}
\end{theorem}

\noindent Here we use the notation $\ll$ introduced in Remark \ref{r22}:  two sides of the direct and inverse inequalities may be considered as real-valued functions of the initial data $(u_0,v_0)$ of the problem \eqref{31}.)

\begin{remark}\label{r38}
The assumption \eqref{310} is will not be needed in the proof of the direct inequality.
We will show in Lemma \ref{l39} below that \eqref{310} is satisfied for  almost all $(a,c,d,r)\in(0,\infty)^4$.
\end{remark}

We proceed in several steps.

\begin{proof}[Proof of the direct inequality in Theorem \ref{t37}]
Fix an arbitrary bounded interval $I$ and $x_0\in\TT$.
Writing
\begin{equation}\label{312}
Z_k^+=
\begin{pmatrix}
z_{k,1}^+\\z_{k,2}^+
\end{pmatrix}
\qtq{and}
Z_k^-=
\begin{pmatrix}
z_{k,1}^-\\z_{k,2}^-
\end{pmatrix}
\end{equation}
for brevity, using the Young inequality,  applying Theorem \ref{t21} (ii), and finally using the relation $\norm{Z_k^{\pm}}\asymp 1$ from Lemma \ref{l32}, we have
\begin{align*}
\int_I&\abs{u(t,x_0)}^2+\abs{v(t,x_0)}^2\ dt
=\int_I\Bigl\lVert\sum_{k\in\ZZ} e^{-ikx_0}\left[c_k^+e^{i\omega_k^+t}Z_k^++c_k^-e^{i\omega_k^-t}Z_k^-\right]\Bigr\rVert^2\ dt\\
&\le 2\int_I\Bigl\lVert\sum_{k\in\ZZ} e^{-ikx_0}c_k^+e^{i\omega_k^+t}Z_k^+\Bigr\rVert^2\ dt
+2\int_I\Bigl\lVert\sum_{k\in\ZZ} e^{-ikx_0}c_k^-e^{i\omega_k^-t}Z_k^-\Bigr\rVert^2\ dt\\
&=2\int_I\Bigl\lVert\sum_{k\in\ZZ} e^{-ikx_0}c_k^+e^{i\omega_k^+t}z_{k,1}^+\Bigr\rVert^2\ dt
+2\int_I\Bigl\lVert\sum_{k\in\ZZ} e^{-ikx_0}c_k^+e^{i\omega_k^+t}z_{k,2}^+\Bigr\rVert^2\ dt\\
&\qquad\qquad +2\int_I\Bigl\lVert\sum_{k\in\ZZ} e^{-ikx_0}c_k^-e^{i\omega_k^-t}z_{k,1}^-\Bigr\rVert^2\ dt
+2\int_I\Bigl\lVert\sum_{k\in\ZZ} e^{-ikx_0}c_k^-e^{i\omega_k^-t}z_{k,2}^-\Bigr\rVert^2\ dt\\
&\ll \sum_{k\in\ZZ}\sum_{j=1}^2\abs{e^{-ikx_0}c_k^+z_{k,j}^+}^2+\abs{e^{-ikx_0}c_k^-z_{k,j}^-}^2\\
&=\sum_{k\in\ZZ}\abs{c_k^+}^2\norm{Z_k^+}^2+\abs{c_k^-}^2\norm{Z_k^-}^2\\
&\asymp \sum_{k\in\ZZ}\abs{c_k^+}^2+\abs{c_k^-}^2.
\end{align*}
Using \eqref{38} we conclude that
\begin{equation*}
\int_I\abs{u(t,x_0)}^2+\abs{v(t,x_0)}^2\ dt
\ll E.\qedhere
\end{equation*}
\end{proof}

\begin{proof}[Proof of a weakened version of the inverse inequality in Theorem \ref{t37}]
We fix a positive number $K$ whose value will be precised later, and we consider only solutions of the form \eqref{36} satisfying $c_k^{\pm}=0$ whenever $\abs{k}\le K$, i.e., functions of the form
\begin{equation}\label{313}
\begin{pmatrix}
u\\ v
\end{pmatrix}
(t,x)=\sum_{\substack{k\in\ZZ\\ \abs{k}>K}}\left(c_k^+e^{i\omega_k^+t}Z_k^++c_k^-e^{i\omega_k^-t}Z_k^-\right)e^{-ikx}
\end{equation}
with square summable complex coefficients $c_k^{\pm}$.
Using the Young inequality and applying Theorem \ref{t25} we have
\begin{align*}
\int_I\abs{u(t,x_0)}^2+&\abs{v(t,x_0)}^2\ dt
=\int_I\Bigl\lVert\sum_{\substack{k\in\ZZ\\ \abs{k}>K}} e^{-ikx_0}\left[c_k^+e^{i\omega_k^+t}Z_k^++c_k^-e^{i\omega_k^-t}Z_k^-\right]\Bigr\rVert^2\ dt\\
&\ge \frac{1}{2}\int_I\Bigl\lVert\sum_{\substack{k\in\ZZ\\ \abs{k}>K}} e^{-ikx_0}\left[c_k^+e^{i\omega_k^+t}Z^++c_k^-e^{i\omega_k^-t}Z^-\right]\Bigr\rVert^2\ dt \\
&\qquad\qquad -\int_I\Bigl\lVert\sum_{\substack{k\in\ZZ\\ \abs{k}>K}} e^{-ikx_0}\left[c_k^+e^{i\omega_k^+t}(Z^+-Z_k^+)+c_k^-e^{i\omega_k^-t}(Z^--Z_k^-)\right]\Bigr\rVert^2\ dt\\
&\ge \frac{1}{2}\int_I\Bigl\lVert\sum_{\substack{k\in\ZZ\\ \abs{k}>K}} e^{-ikx_0}\left[c_k^+e^{i\omega_k^+t}Z^++c_k^-e^{i\omega_k^-t}Z^-\right]\Bigr\rVert^2\ dt \\
&\qquad\qquad -\alpha(I,K)\delta_K^2\sum_{\substack{k\in\ZZ\\ \abs{k}>K}}\left(\abs{c_k^+}^2+\abs{c_k^-}^2\right)
\end{align*}
with
\begin{equation*}
\delta_K:=\sup\set{\norm{Z^+-Z_k^+}, \norm{Z^--Z_k^-}\ :\ \abs{k}>K}.
\end{equation*}

Since $Z^+$ and $Z^-$ are orthogonal, and  $\norm{Z^{\pm}}\ge 2ac$ by Lemma \ref{l32}, it follows that
\begin{align*}
\int_I\abs{u(t,x_0)}^2&+\abs{v(t,x_0)}^2\ dt\\
&\ge \frac{1}{2}\int_I\Bigl\lVert\sum_{\substack{k\in\ZZ\\ \abs{k}>K}} e^{-ikx_0}c_k^+e^{i\omega_k^+t}Z^+\Bigr\rVert^2+\Bigl\lVert\sum_{\substack{k\in\ZZ\\ \abs{k}>K}} e^{-ikx_0}c_k^-e^{i\omega_k^-t}Z^-\Bigr\rVert^2\ dt \\
&\qquad\qquad -\alpha\delta_K^2\sum_{\substack{k\in\ZZ\\ \abs{k}>K}}\left(\abs{c_k^+}^2+\abs{c_k^-}^2\right)\\
&\ge 2a^2c^2\int_I\Bigl\lVert\sum_{\substack{k\in\ZZ\\ \abs{k}>K}} e^{-ikx_0}c_k^+e^{i\omega_k^+t}\Bigr\rVert^2+\Bigl\lVert\sum_{\substack{k\in\ZZ\\ \abs{k}>K}} e^{-ikx_0}c_k^-e^{i\omega_k^-t}\Bigr\rVert^2\ dt \\
&\qquad\qquad -\alpha(I,K)\delta_K^2\sum_{\substack{k\in\ZZ\\ \abs{k}>K}}\left(\abs{c_k^+}^2+\abs{c_k^-}^2\right).
\end{align*}

Assume that $ad\ne 1$.
If $K$ is sufficiently large, say $K\ge K_0$, then by Corollary \ref{c36} we may apply Theorem \ref{21} (iii) to conclude with some constant $\beta(I,K)>0$ the estimate
\begin{align*}
\int_I\abs{u(t,x_0)}^2+\abs{v(t,x_0)}^2dt&\ge \beta(I,K)\sum_{\substack{k\in\ZZ\\ \abs{k}>K}} \left(\abs{e^{-ikx_0}c_k^+}^2+\abs{e^{-ikx_0}c_k^-}^2\right)\\
&\qquad\qquad 
-\alpha(I,K)\delta_K^2\sum_{\substack{k\in\ZZ\\ \abs{k}>K}}\left(\abs{c_k^+}^2+\abs{c_k^-}^2\right)\\
&=\left(\beta(I,K)-\alpha(I,K)\delta_K^2\right)\sum_{\substack{k\in\ZZ\\ \abs{k}>K}}\left(\abs{c_k^+}^2+\abs{c_k^-}^2\right).
\end{align*}
Let us observe that if $K\ge K_0$, then $\beta(I,K)\ge\beta(I,K_0)$ $\alpha(I,K)\le\alpha(I,K_0)$ because we have less complex families to consider for $K$ than for $K_0$.
Therefore we have
\begin{equation*}
\int_I\abs{u(t,x_0)}^2+\abs{v(t,x_0)}^2dt
\ge\left(\beta(I,K_0)-\alpha(I,K_0)\delta_K^2\right)\sum_{\substack{k\in\ZZ\\ \abs{k}>K}}\left(\abs{c_k^+}^2+\abs{c_k^-}^2\right)
\end{equation*}
for every $K\ge K_0$ and for all complex families $(c_k^{\pm})_{\abs{k}>K}$.

Since $\delta_K\to 0$ as $K\to\infty$, we may choose a sufficiently large $K\ge K_0$ such that $\beta(I,K_0)-\alpha(I,K_0)\delta_K^2>0$.
Then using \eqref{38} we conclude that under the assumption $ad\ne 1$ the inverse inequality holds for all  functions of the form \eqref{313}.

If $ad=1$, then by Corollary \ref{c36} we may repeat the last reasoning for every bounded interval of length
\begin{equation*}
\abs{I}>\frac{2\pi (c+1)}{r}.\qedhere
\end{equation*}
\end{proof}

\begin{proof}[Proof of the inverse inequality in Theorem \ref{t37}]
Thanks to our assumption \eqref{310} and to  Corollary \ref{c36} we may apply Theorem \ref{t26} to infer the inverse inequality from the weakened inverse inequality, established in the preceding step.
\end{proof}

\begin{proof}[Proof of the lack of the inverse inequality if $ad=1$ and $\abs{I}<\frac{2\pi (c+1)}{r}$]
For any fixed positive integer $K$, by Corollary \ref{c36} and Remark \ref{r24} (i) we may find square summable complex numbers $c_k^-$ satisfying the inequality
\begin{equation*}
\int_I\abs{\sum_{k\in\ZZ}c_k^-e^{-ikx_0}e^{i\omega_k^-t}}^2\ dt<\frac{1}{K}\sum_{k\in\ZZ}\abs{c_k^-}^2.
\end{equation*}
Moreover, since the removal of finitely many frequencies does not alter the upper density, we may also assume that
\begin{equation*}
c_k^-=0\qtq{whenever}\abs{k}\le K.
\end{equation*}
Then, applying the Young inequality and then Theorem \ref{t25} (ii) with $J=1$ and 
\begin{equation*}
\delta_K:=\sup\set{\norm{Z_k^--Z^-}\ :\ \abs{k}>K},
\end{equation*}
we have the following estimate:
\begin{align*}
\int_I\norm{\sum_{k\in\ZZ}c_k^-e^{i\omega_k^-t}Z_k^-e^{-ikx_0}}^2\ dt
&\le 2\int_I\norm{\sum_{k\in\ZZ}c_k^-e^{i\omega_k^-t}Z^-e^{-ikx_0}}^2\ dt\\
&\qquad +2\int_I\norm{\sum_{k\in\ZZ}c_k^-e^{i\omega_k^-t}(Z_k^--Z^-)e^{-ikx_0}}^2\ dt\\
&\le \left(\frac{2\norm{Z^-}^2}{K}+2\alpha\delta_K^2\right)\sum_{k\in\ZZ}\abs{c_k^-}^2.
\end{align*}

Since the expression in front of the sum tends to zero as $K\to\infty$, this estimate shows that the inverse inequality \eqref{311} does not hold. 
\end{proof}

Theorem \ref{t12} (iv) now follows from Theorem \ref{t37} and from the following lemma, showing that the hypothesis (3.10)
holds for  almost all choices of the parameters $a,c,d,r$:

\begin{lemma}\label{l39}
For almost every quadruple $(a,c,d,r)\in(0,\infty)^4$,
we have $\omega_k^+\ne\omega_n^+$ and $\omega_k^-\ne\omega_n^-$ whenever $k\ne n$.
\end{lemma}

\begin{proof}
We recall that $\omega_0^+=\omega_0^-=0$.
If $k\ne 0$, and $\omega_k^+=0$ or $\omega_k^-=0$, then we infer from the determinant defining $\omega_k^{\pm}$ that $(1-ad)k^6-rk^4=0$, i.e., $r=(1-ad)k^2$.
Hence for any fixed $(a,d)\in(0,\infty)^2$ there are at most countably many values of $r>0$ such that $\omega_k^+=0$ or $\omega_k^-=0$ for some $k\in\ZZ^*$.
Applying Fubini's theorem hence we infer that for almost every $(a,c,d,r)\in(0,\infty)^4$, we have $\omega_k^+\ne \omega_0^+$ and $\omega_k^-\ne \omega_0^-$ for all $k\in\ZZ^*$.

Henceforth we consider only nonzero indices $n$ and $k$.
Setting
\begin{equation*}
I_k:=4acdk^4+[r+(c-1)k^2]^2
=r^2+(2c-2)k^2r+[(c-1)^2+4acd]k^4
\end{equation*}
for brevity, in case $\omega_k^+=\omega_n^+$  we deduce from \eqref{33} that
\begin{equation}\label{314}
(c+1)(k^3-n^3)-(k-n)r
=n\sqrt{I_n}-k\sqrt{I_k}.
\end{equation}
Taking the square of both sides hence
we get
\begin{equation}\label{315}
nk\sqrt{I_n}\sqrt{I_k}
=nkr^2
+\left[2c(n^4+k^4)-(c+1)nk(n^2+k^2)\right]r+\alpha_1
\end{equation}
with a constant $\alpha_1$ not depending on $r$.
Taking its square again, we get after some simplification the equation
\begin{equation}\label{316}
2c(n-k)(n^3-k^3)r^3+\alpha_2r^2+\alpha_3r+\alpha_4=0
\end{equation}
with suitable constants $\alpha_2, \alpha_3, \alpha_4$, independent of $r$.

Similarly, in case $\omega_k^-=\omega_n^-$  we deduce from \eqref{33} instead of \eqref{314} the equality
\begin{equation*}
(c+1)(k^3-n^3)-(k-n)r
=k\sqrt{I_k}-n\sqrt{I_n},
\end{equation*}
and then the same equalities \eqref{315} and \eqref{316}.

For any fixed $(a,c,d)\in(0,\infty)^3$ and for any $(k,n)\in \ZZ^2$ with $k\ne n$ the polynomial equation \eqref{316} vanishes for at most three values of $r$.
Therefore for any fixed  $(a,c,d)\in(0,\infty)^3$, $\omega_k^+\ne\omega_n^+$ and $\omega_k^-\ne\omega_n^-$ for all $k\ne n$ for all but countable many values of $r$.
The lemma follows by applying Fubini's theorem.
\end{proof}

\section{Pointwise controllability}\label{s4}

In this section we study the pointwise controllability of the  system
\begin{equation}\label{41}
\begin{cases}
u_t+u_{xxx}+av_{xxx}=f(t)\delta_{x_0}\qtq{in}\RR\times\TT,\\
v_t+\frac{r}{c}v_x+\frac{1}{c}v_{xxx}+\frac{d}{c}u_{xxx}=g(t)\delta_{x_0}\qtq{in}\RR\times\TT,\\
u(0)=u_0\qtq{and}v(0)=v_0,\end{cases}
\end{equation}
where $a,c,d,r$ are given positive constants, $\delta_{x_0}$ denotes the Dirac delta function centered in a given point $x_0\in\TT$, and $f,g$ are the control functions.
We will prove the following
\begin{theorem}\label{t41}
Fix $x_0\in\TT$ arbitrarily, and choose $a,c,d,r$ such that  \eqref{310} be satisfied. Set
\begin{equation*}
T_0:=
\begin{cases}
0&\text{if $ad\ne 1$,}\\
2\pi (c+1)/r&\text{if $ad=1$.}
\end{cases}
\end{equation*}

Given $T>T_0$ arbitrarily, for every  $(u_0,v_0), (u_T,v_T)\in H$ there exist control functions $f,g\in L^2_{\text{loc}}(\RR)$ such that the solution of \eqref{41} satisfies the final conditions
\begin{equation*}
u(T)=u_T\qtq{and} v(T)=v_T.
\end{equation*}
\end{theorem}

\begin{remark}\label{r42}
In case $ad=1$ the system is not controllable for any $T<T_0$.
This follows from Theorem \ref{t37} and from a general theorem of duality between observability and controllability; see, e.g., \cite{DolRus1977}.
\end{remark}

By a general argument of control theory, it suffices to consider the special case of null controllability, i.e., the case where $u_T=v_T=0$.
We prove this  by applying the Hilbert Uniqueness Method (HUM) of  J.-L. Lions \cite{Lions1988} as follows.
Fix $(u_0,v_0)\in H$ arbitrarily.
Choose $(\varphi_0,\psi_0)\in D(A)$
arbitrarily, and solve the homogeneous adjoint system
\begin{equation}\label{42}
\begin{cases}
\varphi_t + \varphi_{xxx} + \frac{d}{c}\psi_{xxx}=0  \qtq{in}\RR\times\TT,\\
\psi_t  +\frac{r}{c}\psi_x+a\varphi_{xxx} +\frac{1}{c}\psi_{xxx} =0  \qtq{in}\RR\times\TT,\\
\varphi(0)=\varphi_0\qtq{and}\psi(0)=\psi_0.
\end{cases}
\end{equation}

Then solve the non-homogeneous problem with zero final states (instead of the non-zero initial states  in \eqref{41}) by applying the controls
\begin{equation}\label{43}
f:=-\varphi(\cdot,x_0)\qtq{and}g:=-\psi(\cdot,x_0):
\end{equation}
\begin{equation}\label{44}
\begin{cases}
u_t+u_{xxx}+av_{xxx}=-\varphi(\cdot,x_0)\delta_{x_0}\qtq{in}\RR\times\TT,\\
v_t+\frac{r}{c}v_x+\frac{1}{c}v_{xxx}+\frac{d}{c}u_{xxx}=-\psi(\cdot,x_0)\delta_{x_0}\qtq{in}\RR\times\TT,\\
u(T)=v(T)=0.
\end{cases}
\end{equation}
If $(u(0),v(0))$ happens to be equal to $(u_0,v_0)$, then the solution of \eqref{41} with the controls \eqref{43} will satisfy the final conditions of the theorem by the uniqueness (to be proven below) of the solutions of \eqref{41} and \eqref{44}.

Therefore the theorem will be completed by showing that the range of the map
\begin{equation*}
\Lambda(\varphi_0,\psi_0):=(u(0),v(0))
\end{equation*}
contains $H$.

Now we make this approach precise.
We start by defining the solutions of \eqref{41}.
We begin with a formal computation.
If $(u,v)$ solves \eqref{41} and $(\varphi,\psi)$ solves \eqref{42}, then we obtain for each fixed $T\in\RR$ the following equalities by integration by parts:
\begin{align*}
0
&=\int_0^T\int_{\TT} u\left(\varphi_t+\varphi_{xxx}+\frac{d}{c}\psi_{xxx}\right)\ dx\ dt \\
&\qquad\qquad+\int_0^T\int_{\TT}v\left(\psi_t+\frac{r}{c}\psi_x+a\varphi_{xxx}+\frac{1}{c}\psi_{xxx}\right)\ dx\ dt \\
&=\left[\int_{\TT}u\varphi+v\psi\ dx\right]_0^T-\int_0^T\int_{\TT} \left(u_t+u_{xxx}+av_{xxx}\right)\varphi\ dx\ dt \\
&\qquad\qquad-\int_0^T\int_{\TT} \left(v_t+\frac{r}{c}v_x+\frac{1}{c}v_{xxx}+\frac{d}{c}u_{xxx}\right)\psi\ dx\ dt \\
&=\left[\int_{\TT}u\varphi+v\psi\ dx\right]_0^T
-\int_0^Tf(t)\varphi(t,x_0)+g(t)\psi(t,x_0)\ dt.
\end{align*}
This may be rewritten in the form
\begin{multline}\label{45}
\left((u(T),v(T)),(\varphi(T),\psi(T))\right)_H\\
=\left((u_0,v_0),(\varphi_0,\psi_0)\right)_H
+\left((f,g),(\varphi(\cdot,x_0),\psi(\cdot,x_0))\right)_{G_T},
\end{multline}
where $G_T$ denotes the Hilbert space $L^2(0,T)\times L^2(0,T)$ for $T>0$ and $L^2(T,0)\times L^2(T,0)$ for $T<0$.
This identity leads to the following definition:

\begin{definition}
By a solution of \eqref{41} we mean a function $(u,v)\in C_w(\RR,H)$ satisfying \eqref{45} for all $T\in\RR$ and $(\varphi_0,\psi_0)\in H$.
\end{definition}

The subscript $w$ indicates that $H$ is endowed here with the weak topology.
The definition is justified by the following

\begin{theorem}\label{t43}
Given any initial data $(u_0,v_0)\in H$ and control functions $f,g\in L^2_{\text{loc}}(\RR)$, the system \eqref{41} has a unique solution, and the linear map
\begin{equation*}
(u_0,v_0,f,g)\mapsto (u,v)
\end{equation*}
is continuous for the indicated topologies.
\end{theorem}

\begin{remark}\label{r44}
By the time invariance of the system \eqref{41} the theorem remains valid if we impose final conditions instead of initial conditions; hence it also proves the well posedness of  \eqref{44}.
\end{remark}

\begin{proof}
Denoting the right hand side of \eqref{45} by $L_T(\varphi_0,\psi_0)$ for each fixed $T\in\RR$, it follows from Theorem \ref{t33} and the direct inequality in Theorem \ref{t37} that $L_T$ is a continuous linear functional of $(\varphi_0,\psi_0)$. The same argument ensures the  and $L_T$ is a continuous linear functional of $(\varphi(T),\psi(T))$.
This proves the existence of a unique couple $(u(T),v(T))\in H$ satisfying  \eqref{45}.

The weak continuity of the solution  follows by observing that for any fixed $(\varphi_0,\psi_0)\in H$ the right hand side of \eqref{45} is continuous in $T$ by the continuity of the primitives of Lebesgue integrable real functions.

Using Theorem \ref{t33} again, we may also infer from \eqref{45} for each $T>0$ the estimates
\begin{equation*}
\norm{(u,v)}_{L^{\infty}(-T,T;H)}\le \set{\norm{(u_0,v_0)}_H+\norm{(f,g)}_{L^2(-T,T)}}
\end{equation*}
for all $T>0$; this proves the continuous dependence of the solution on the data.
\end{proof}

\begin{remark}\label{r45}
Since we have only used the direct inequality in Theorem \ref{t37}, Theorem \ref{t43} holds without the assumption \eqref{310} on the eigenvalues.
\end{remark}

Now we turn back to the proof of Theorem \ref{t41}.
It remains to show that the range of the map $\Lambda$ contains $H$.
Indeed, it is a continuous linear map $\Lambda:H\to H$ by Theorems \ref{t33} and \ref{t43}.
Furthermore, we infer from \eqref{45} that
\begin{equation*}
(\Lambda(\varphi_0,\psi_0),(\varphi_0,\psi_0))_H
=\int_0^T\abs{\varphi(t,x_0)}^2+\abs{\psi(t,x_0)}^2\ dt,
\end{equation*}
and then from Theorem \ref{t37} that
\begin{equation*}
(\Lambda(\varphi_0,\psi_0),(\varphi_0,\psi_0))_H
\ge \varepsilon\norm{(\varphi_0,\psi_0)}_H^2
\end{equation*}
for all $(\varphi_0,\psi_0)\in H$ with a suitable positive constant $\varepsilon$.

Applying the Lax--Milgram Theorem we conclude that $\Lambda$ is an isomorphism of $H$ onto itself; in particular, it is onto.

\section{Pointwise stabilizability}\label{s5} 

Following \cite[Chapter 2]{KomLor2005} first we recall some general abstract results.
Consider the evolutionary problem
\begin{equation}\label{51}
U'=\AA U,\quad U(0)=U_0
\end{equation}
(we use  the notation $U'$ for the time derivative of $U$), where
\begin{itemize}
\item[(i)] $\AA$ is a skew-adjoint linear operator in a Hilbert space $\HH$, having a compact resolvent.
\end{itemize}
Then $\AA$ generates a strongly continuous group of automorphisms $e^{t\AA}$ in $\HH$, and for each $U_0\in \HH $ the problem \eqref{51} has a unique
continuous solution $U:\RR\to\HH$, satisfying the equality
\begin{equation*}
\norm{U(t)}=\norm{U_0}\qtq{for all}t\in\RR.
\end{equation*}

Now let $\BB$ be another linear operator, defined on some linear subspace $D(\BB)$ of $\HH$ with  values in another Hilbert space $\GG$, satisfying the following two hypotheses:
\begin{itemize}
\item[(ii)]$D(\AA )\subset D(\BB )$, and there exists a positive constant $c$ such that
\begin{equation*}
\norm{\BB U_0}_{\GG}\le c\norm{\AA U_0}_{\HH}\qtq{for all}U_0\in D(\AA );
\end{equation*}
\item[(iii)] there exist a non-degenerate bounded interval $I$ and a positive constant $c_I$ such that the solutions of
\eqref{51} satisfy the inequalities
\begin{equation*}
\norm{\BB U}_{L^2(I;\GG)} \le c_I\norm{U_0}_{\HH}\qtq{for all}U_0\in D(\AA ).
\end{equation*}
\end{itemize}
The operator $\BB$ is usually called an {\em observability} operator: We may think that we  can observe only $\BB U$ and not the whole solution $U$.

Under the assumptions (i), (ii), (iii) we may define the solutions of the {\em
dual} problem
\begin{equation}
V'=-\AA^* V+\BB^* W,\quad V(0)=V_0,\label{52}
\end{equation}
by transposition, were $\HH'$,  $\GG'$ denote the
dual spaces of $\HH$,  $\GG$,
and $\AA^*$, $\BB^*$ denote the adjoints of $\AA$ and $\BB$.

The operator $\BB^*$ is usually called a {\em controllability} operator: we may think that we  can act on the system by choosing a
{\em control} $W$.

Motivated by a formal computation, by a solution of \eqref{52} we mean a
function  $V:\RR\to \HH'$
satisfying the identity
\begin{equation*}
\langle V(S),U (S)\rangle_{\HH',\HH}=\langle V_0,U _0\rangle_{\HH',\HH}
+\int_0^s\langle W(t),\BB U(t)\rangle_{\GG',\GG}\ dt
\end{equation*}
for all $U _0\in \HH$ and for all
$s\in  \RR$.
Then for every $V_0\in \HH'$ and $W\in L^2_{\rm loc}(\RR;\GG')$, \eqref{52} has  a
unique solution.
Moreover, the function $V:\RR\to \HH'$ is weakly continuous.

The Hilbert Uniqueness Theorem states that if the inverse inequality of (iii) also holds:
\begin{itemize}
\item[(iv)] there exist a bounded interval $I'$ and a constant $c'$ such that
the solutions of
\eqref{51}
satisfy the inequality
\begin{equation*}
\norm{U_0}_{\HH}\le c'\norm{\BB U}_{L^2(I';\GG)}\qtq{for all}U_0\in D(\AA ),
\end{equation*}
\end{itemize}
then the system \eqref{52} is exactly controllable in the following sense:

\begin{theorem}\label{t51}
Assume (i), (ii), (iii), (iv), and let $T>\abs{I'}$ (the length of $I'$).
Then for all $V_0, V_1\in \HH'$  there exists a function $W\in L^2_{\rm loc}(\RR;\GG')$ such that  the solution of  \eqref{52}
satisfies $V(T)=V_1$.
\end{theorem}
\noindent We may of course assume that $W$ vanishes outside the interval $(0,T)$.

Under the assumptions of the theorem we may also construct feedback controls yielding arbitrarily fast decay rates:

\begin{theorem}\label{t52}
Assume (i), (ii), (iii), (iv), and fix $\omega >0$ arbitrarily.
There exists a bounded linear map $F:\HH'\to\GG'$ and a constant $M>0$ such that the problem
\begin{equation}\label{53}
V'=-\AA^*V+\BB^*FV,\qquad V(0)=V_0
\end{equation}
has a unique weakly continuous solution $V:\RR\to\HH'$, and the solutions satisfy the estimates
\begin{equation*}
\norm{V(t)}_{\HH'} \le M\norm{V_0}_{\HH'} e^{-\omega t}
\end{equation*}
for all $V_0\in \HH'$ and $t\ge 0$.
\end{theorem}

Let us observe that Theorems \ref{t51} and \ref{t52} have the same assumptions.
These assumptions have been verified during the proof of Theorem \ref{t11}.
Therefore we may also apply Theorem \ref{t52} for the Gear--Grimshaw system, and Theorem \ref{t13} follows.

\section{Use of one control}\label{s6}

In this section we establish a variant of Theorem \ref{t37} when we observe only one of the functions $u(\cdot,x_0)$ and  $v(\cdot,x_0)$.
Such observations do not allow us to determine completely the initial data in \eqref{31}.
Indeed, if $(u,v)$ solves \eqref{31}, then for any constant $c$ the couple $(u,v+c)$ also solves \eqref{31} with some other initial data, so that the observation of the component $u$ may allow to determine $v$ up to an arbitrary additive constant.
An analogous situation occurs by observing $v$.

Theorem \ref{t62} and Corollary \ref{c63} below will show that up to this indeterminacy the determination of the solutions is possible by observing only one component.
Since the solutions of \eqref{31} are given by the formulas
\begin{equation}\label{61}
\begin{cases}
u(t,x)=\sum_{k\in\ZZ}\left(c_k^+e^{i\omega_k^+t}z_{k,1}^++c_k^-e^{i\omega_k^-t}z_{k,1}^-\right)e^{-ikx},\\
v(t,x)=\sum_{k\in\ZZ}\left(c_k^+e^{i\omega_k^+t}z_{k,2}^++c_k^-e^{i\omega_k^-t}z_{k,2}^-\right)e^{-ikx}
\end{cases}
\end{equation}
where we use the notation \eqref{312},
we need an Ingham type theorem for the family $\set{\omega_k^{\pm}\ :\ k\in\ZZ}$.
It does not have a uniform gap, because $\omega_0^+=\omega_0^-=0$ and because $\omega_k^+$ may be close to $\omega_n^-$ for many couples $(k,n)$, but it satisfies the weakened gap condition of Theorem \ref{t23}
with $M=2$.

Given a positive number $\varepsilon$ we consider in the set
\begin{equation*}
\Omega:=\set{\omega_k^{\pm}\ :\ k\in\ZZ}
\end{equation*}
the following equivalence relation: $x\sim y$ if $x=y$, or if there exists a finite sequence $x_1,\ldots, x_n$ in $\Omega$ such that $x=x_1$, $x_n=y$, and
\begin{equation*}
\abs{x_{j+1}-x_j}<\varepsilon\qtq{for} j=1,\ldots, n-1.
\end{equation*}

\begin{lemma}\label{l61}
\mbox{}

\begin{enumerate}[\upshape (i)]
\item We have
\begin{equation*}
\abs{z_{k,j}^{\pm}}\asymp 1,\quad j=1,2.
\end{equation*}
\item For almost every quadruple $(a,c,d,r)\in(0,\infty)^4$, we have
\begin{equation}\label{62}
\omega_k^+\ne\omega_n^+
\qtq{and}
\omega_k^-\ne\omega_n^-
\qtq{whenever}k\ne n,
\end{equation}
and
\begin{equation}\label{63}
\omega_k^+\ne\omega_n^-
\qtq{for all}k,n\in\ZZ,
\qtq{except if}k=n=0.
\end{equation}
\item Assume \eqref{62} and \eqref{63}.
If $\varepsilon$ sufficiently small, then each equivalence class of $\Omega$ has one or two elements.
Moreover, if it has two elements, then one of them belongs to $\set{\omega_k^+\ :\ k\in\ZZ}$, and the other one belongs to $\set{\omega_k^-\ :\ k\in\ZZ}$.
\end{enumerate}
\end{lemma}

\begin{proof}
(i) This  follows from the explicit expression \eqref{34} of these vectors.
\medskip

(ii) In view of Lemma \ref{l39} we only need to consider the property \eqref{63}.
For this we adapt the proof of Lemma \ref{l39} as follows.

If $\omega_k^+=\omega_n^-$ for some $k,n\in\ZZ$, then  we deduce from \eqref{33} the equality
\begin{equation*}
(c+1)(k^3-n^3)-(k-n)r
=-n\sqrt{I_n}-k\sqrt{I_k}
\end{equation*}
instead of \eqref{314}, and then the equality
\begin{equation*}
-nk\sqrt{I_n}\sqrt{I_k}
=nkr^2+\left[2c(n^4+k^4)-(c+1)nk(n^2+k^2)\right]r+\alpha_1
\end{equation*}
instead of \eqref{315}.
Taking its square we arrive at the same equation \eqref{316} as before.
\medskip

(iii) For any fixed $\varepsilon>0$, by Lemmas \ref{l34} and \ref{l35} there exists a sufficiently large positive integer $K$ such that
\begin{equation*}
\abs{\omega_k^+-\omega_n^+}\ge 2\varepsilon
\qtq{and}
\abs{\omega_k^--\omega_n^-}\ge 2\varepsilon
\end{equation*}
whenever $\abs{k}, \abs{n}>K$ and $k\ne n$.
Then each equivalence class in the restricted set
\begin{equation*}
\set{\omega_k^{\pm}\ :\ \abs{k}>K}
\end{equation*}
has at most two elements.
Indeed, if two elements are equivalent, then they have to belong to the different families $\set{\omega_k^+}$ and $\set{\omega_k^-}$, say $\abs{\omega_k^+-\omega_n^-}<\varepsilon$.
Then we infer from our choice of $K$ that
\begin{equation*}
\abs{\omega_k^+-\omega_m^-}>\varepsilon\qtq{for all}m\ne n
\end{equation*}
and
\begin{equation*}
\abs{\omega_m^+-\omega_n^-}>\varepsilon\qtq{for all}m\ne k,
\end{equation*}
so that no other exponent is equivalent to $\omega_k^+$ and $\omega_n^-$.

This property remains valid if we change $\varepsilon$ to a smaller positive value.
Indeed, each one-point equivalence class remains the same, while the others either remain the same or they split into two one-point equivalence classes.

Next we observe that each element of $\Omega$ is isolated.
Therefore, if we diminish $\varepsilon$ so as to satisfy the finite number of inequalities
\begin{align*}
&\dist(\omega_k^+,\Omega\setminus\set{\omega_k^+})>\varepsilon\qtq{for}k=0, \pm 1,\ldots, \pm K,
\intertext{and}
&\dist(\omega_k^-,\Omega\setminus\set{\omega_k^-})>\varepsilon\qtq{for}k=0, \pm 1,\ldots, \pm K,
\end{align*}
then $\set{\omega_k^+}$ and $\set{\omega_k^-}$ will be one-point equivalence classes for each $k=0,\pm 1,\ldots, \pm K$.
(The equality $\omega_0^+=\omega_0^-=0$ does not contradict these properties because $\omega_0^+$ and $\omega_0^-$ are the same element of $\Omega$.)
\end{proof}

Under the conditions of Lemma \ref{l61} we may rewrite the solutions  \eqref{61} of \eqref{31} as follows.
Whenever $\omega_k^+\sim\omega_n^-$ and $(\omega_k^+,\omega_n^-)\ne (0,0)$, we rewrite the corresponding terms
\begin{equation*}
c_k^+e^{i\omega_k^+t}z_{k,j}^+e^{-ikx}+c_n^-e^{i\omega_n^-t}z_{n,j}^-e^{inx}
\end{equation*}
in the form
\begin{equation}\label{64}
a_{k,j}^+(x)z_{k,j}^+e^{i\omega_k^+t}+a_{n,j}^-(x)z_{n,j}^-\frac{e^{i\omega_k^+t}-e^{i\omega_n^-t}}{\omega_k^+-\omega_n^-}
\end{equation}
for $j=1,2$ with suitable new coefficients, and we write
\begin{equation}\label{65}
e_k^+(t):=e^{i\omega_k^+t},
\quad e_n^-(t):=\frac{e^{i\omega_k^+t}-e^{i\omega_n^-t}}{\omega_k^+-\omega_n^-}.
\end{equation}
(By Lemma \ref{l61} (iii) we do not have to consider higher-order divided differences when applying Theorem \ref{t23}.)
For all other exponents (in particular, for $\omega_0^+=\omega_0^-=0$) we set
\begin{equation*}
e_k^{\pm}(t):=e^{i\omega_k^{\pm}t}
\qtq{and}
a_{k,j}^{\pm}(x):=c_k^{\pm}.
\end{equation*}
(These coefficients are in fact independent of $x$.)
We have thus instead of \eqref{61} the following representation:
\begin{equation}\label{66}
\begin{cases}
u(t,x)=\sum_{k\in\ZZ}\left[a_{k,1}^+(x)z_{k,1}^+e_k^+(t)
+a_{k,1}^-(x)z_{k,1}^-e_k^-(t)\right]e^{-ikx},\\
v(t,x)=\sum_{k\in\ZZ}\left[a_{k,2}^+(x)z_{k,2}^+e_k^+(t)
+a_{k,2}^-(x)z_{k,2}^-e_k^-(t)\right]e^{-ikx}.
\end{cases}
\end{equation}

Using this representation we may state our theorem, where we use the notation $\ZZ^*:=\ZZ\setminus\set{0}$:

\begin{theorem}\label{t62}
Assume \eqref{62} and \eqref{63}, and fix $x_0\in\TT$ arbitrarily.
Then the solutions of \eqref{31} have the following properties:

\begin{enumerate}[\upshape (i)]
\item the direct inequalities
\begin{align*}
&\int_I\abs{u(t,x_0)}^2\ dt
\ll \abs{a_{0,1}^++a_{0,1}^-}^2+\sum_{k\in\ZZ^*}\left(\abs{a_{k,1}^+(x_0)}^2+\abs{a_{k,1}^-(x_0)}^2\right)
\intertext{and}
&\int_I\abs{v(t,x_0)}^2\ dt
\ll \abs{a_{0,2}^+-a_{0,2}^-}^2+\sum_{k\in\ZZ^*}\left(\abs{a_{k,2}^+(x_0)}^2+\abs{a_{k,2}^-(x_0)}^2\right)
\end{align*}
hold for all non-degenerate bounded intervals $I$.
\item if $ad\ne 1$, then the inverse inequalities
\begin{equation}\label{67}
\abs{a_{0,1}^++a_{0,1}^-}^2+\sum_{k\in\ZZ^*}\left(\abs{a_{k,1}^+(x_0)}^2+\abs{a_{k,1}^-(x_0)}^2\right)\ll \int_I\abs{u(t,x_0)}^2\ dt
\end{equation}
and
\begin{equation}\label{68}
\abs{a_{0,2}^+-a_{0,2}^-}^2+\sum_{k\in\ZZ^*}\left(\abs{a_{k,2}^+(x_0)}^2+\abs{a_{k,2}^-(x_0)}^2\right)\ll \int_I\abs{v(t,x_0)}^2\ dt
\end{equation}
also hold for all non-degenerate bounded intervals $I$.
\item if $ad=1$, then  \eqref{67} and \eqref{68} hold for all bounded intervals $I$ of length $\abs{I}>2\pi (c+1)/r$, and they fail if $\abs{I}<2\pi (c+1)/r$.
\end{enumerate}
\end{theorem}

\noindent The notation $\ll$ corresponds to Remark \ref{r22}: in each of the four above estimates the two sides may be considered as real-valued functions of the complex families $(a_{k,1}^+)$, $(a_{k,1}^-)$, $(a_{k,2}^+)$ and $(a_{k,2}^-)$.

\begin{proof}[Proof of Theorem \ref{t62}]
First we deduce from Corollary \ref{c36} and from the subadditivity relations
\begin{equation*}
D^+(\set{\omega_k^-})
\le D^+(\set{\omega_k^{\pm}})
\le D^+(\set{\omega_k^+})+D^+(\set{\omega_k^-})
\end{equation*}
(see \eqref{22}) that
\begin{equation*}
D^+(\set{\omega_k^{\pm}})=
\begin{cases}
0&\text{if $ad\ne 1$,}\\
\frac{c+1}{r}&\text{if $ad=1$.}
\end{cases} 
\end{equation*}
Using this and Lemma \ref{l61}, the theorem follows by  applying Theorem \ref{t23} with $M=2$.
For $k=0$ we also use the relations (see \eqref{35} and \eqref{65})
\begin{equation*}
a_{0,1}^+z_{0,1}^+e_0^+(t)+a_{0,1}^-z_{0,1}^-e_0^-(t)
=2ac(a_{0,1}^++a_{0,1}^-)
\end{equation*}
and
\begin{equation*}
a_{0,2}^+z_{0,2}^+e_0^+(t)+a_{0,2}^-z_{0,2}^-e_0^-(t)
=\sqrt{4acd}(a_{0,2}^+-a_{0,2}^-).\qedhere
\end{equation*}
\end{proof}

We infer from Theorem \ref{t62} the following uniqueness property:

\begin{corollary}\label{c63}
Assume \eqref{62} and \eqref{63}, and
set
\begin{equation*}
T_0:=
\begin{cases}
0&\text{if $ad\ne 1$,}\\
2\pi (c+1)/r&\text{if $ad=1$.}
\end{cases}
\end{equation*}
Fix $x_0\in\TT$ arbitrarily, consider the solution of \eqref{31} and an interval $I$ of length $\abs{I}>T_0$.

\begin{enumerate}[\upshape (i)]
\item If $u(t,x_0)=0$ for all $t\in I$, then  $u=0$ and $v$ is an arbitrary constant function.
\item If $v(t,x_0)=0$ for all $t\in I$, then  $v=0$ and $u$ is an arbitrary constant function.
\end{enumerate}
\end{corollary}

\begin{proof}
If $u(t,x_0)=0$ for all $t\in I$, then we infer from the estimate of Theorem \ref{t62}
the equalities
\begin{equation*}
a_{0,1}^++a_{0,1}^-=0,
\qtq{and}
a_{k,1}^{\pm}=0
\qtq{for all}
k\in\ZZ^*.
\end{equation*}
In view of \eqref{64} this is equivalent to the relations
\begin{equation*}
c_0^++c_0^-=0,
\qtq{and}
c_k^{\pm}=0
\qtq{for all}
k\in\ZZ^*.
\end{equation*}
Using \eqref{35} we conclude that
\begin{equation*}
u(t,x)=0
\end{equation*}
and
\begin{equation*}
v(t,x)=\sqrt{4acd}(c_0^+-c_0^-),
\end{equation*}
i.e., $u=0$ and $v$ is an arbitrary constant function.

Similarly, if $u(t,x_0)=0$ for all $t\in I$, then we obtain that
\begin{equation*}
c_0^+-c_0^-=0,
\qtq{and}
c_k^{\pm}=0
\qtq{for all}
k\in\ZZ^*.
\end{equation*}
This implies that $v=0$ and $u$ is an arbitrary constant function.
\end{proof}

We end this paper by proving two variants of Theorem \ref{t41} where we apply only one control.
Let us observe that if $f=0$ in \eqref{41}, then $\int_{\TT}u(t,x)\ dx$ does not depend on $t\in\RR$ because
\begin{equation*}
\frac{d}{dt}\int_{\TT}u\ dx
=-\int_{\TT}u_{xxx}+av_{xxx}\ dx
=0
\end{equation*}
by integration by parts.
It follows that by using only the control function $g$  in \eqref{41}, if a state $(u_0,v_0)\in H$ may be driven to $(u_T,v_T)\in H$ in time $T$, then
\begin{equation}\label{69}
\int_{\TT}u_0\ dx=\int_{\TT}u_T\ dx.
\end{equation}

Similarly, if $g=0$ in \eqref{41}, then $\int_{\TT}v(t,x)\ dx$ does not depend on $t\in\RR$ because
\begin{equation*}
\frac{d}{dt}\int_{\TT}v\ dx
=-\int_{\TT}\frac{r}{c}v_x+\frac{1}{c}v_{xxx}+\frac{d}{c}u_{xxx}\ dx
=0.
\end{equation*}
It follows that by using only the control function $f$  in \eqref{41}, if a state $(u_0,v_0)\in H$ may be driven to $(u_T,v_T)\in H$ in time $T$, then
\begin{equation}\label{610}
\int_{\TT}v_0\ dx=\int_{\TT}v_T\ dx.
\end{equation}

Applying the method of ``contr\^olabilit\'e exacte \'elargie'' of Lions \cite[p. 95]{Lions1988}, we prove that these conditions are also sufficient for the controllability if the time is large enough.

\begin{theorem}\label{t64}
Fix $x_0\in\TT$ arbitrarily, and choose $a,c,d,r$ satisfying \eqref{310}, and set
\begin{equation*}
T_0:=
\begin{cases}
0&\text{if $ad\ne 1$,}\\
2\pi (c+1)/r&\text{if $ad=1$.}
\end{cases}
\end{equation*}
Furthermore, fix $T>T_0$ and $(u_0,v_0), (u_T,v_T)\in H$ arbitrarily.

\begin{enumerate}[\upshape (i)]
\item There exists a control function $f\in L^2_{\text{loc}}(\RR)$ such that the solution of \eqref{41} with $g=0$ satisfies the final conditions
\begin{equation*}
u(T)=u_T\qtq{and} v(T)=v_T
\end{equation*}
if and only if \eqref{610} is satisfied.
\item There exists a control function $g\in L^2_{\text{loc}}(\RR)$ such that the solution of \eqref{41} with $f=0$ satisfies the final conditions
\begin{equation*}
u(T)=u_T\qtq{and} v(T)=v_T
\end{equation*}
if and only if \eqref{69} is satisfied.
\end{enumerate}
\end{theorem}

\begin{proof}
The two cases being analogous, we only consider (i).
We have already shown the necessity of the condition \eqref{610}.
It remains to prove the null controllability of all initial data $(u_0,v_0)$ belonging to the closed linear subspace
\begin{equation*}
\tilde H:=\set{(u_0,v_0)\in H\ :\ \int_{\TT}u_0\ dx=0}.
\end{equation*}
For this we repeat the proof of Theorem \ref{t41} by modifying the definition of the operator $\Lambda$: in \eqref{43} and \eqref{44} we change $g=-\psi(\cdot,x_0)$ to $0$.
We obtain a continuous linear map $\Lambda:\tilde H\to\tilde H$, and the hypotheses of the Lax--Milgram theorem are satisfied by Theorem \ref{t62}.
\end{proof}

\begin{remark}\label{r65}
It follows from Theorem \ref{t62} that the value $T_0$ is optimal for the validity of Corollary \ref{c63} and Theorem \ref{t64}.
\end{remark}

\noindent\textbf{Acknowledgments:} The authors wish to thank the referee for his/her valuable comments which improved this paper. R. A. Capistrano-Filho was partially supported by CNPq (Brazil) by the grants 306475/2017-0, 408181/2018-4  and Propesq (UFPE)- Edital Qualis A, V. Komornik was supported by ``Le Réseau Franco-Brésilien en Mathématiques'' and by the grant NSFC No.11871348, and A. F. Pazoto was partially supported by CNPq (Brazil).


\begin{thebibliography}{99}

\bibitem{BaiKomLor111} C. Baiocchi, V. Komornik and P. Loreti, {\em Ingham-Beurling type theorems with weakened gap conditions},
Acta Math. Hungar. 97 (1-2) (2002), 55--95.

\bibitem{Beu}
A. Beurling,
Interpolation for an interval in $\RR^1$,
in {\em The collected works of Arne Beurling. Vol. 2. Harmonic analysis} (eds. L. Carleson, P. Malliavin, J. Neuberger and J. Wermer)
Contemporary Mathematicians. Birkhäuser Boston, Inc., Boston, MA, 1989.

\bibitem{BonPonSauTom1992} J. L. Bona, G. Ponce, J.-C. Saut and M. M. Tom, \emph{A model system for strong interaction between internal solitary waves},
Comm. Math. Physics 143 (1992), 287--313.

\bibitem{CaZu1996} C. Castro and E. Zuazua, {\em Une remarque sur les séries de Fourier non-harmoniques et son application à la contrôlabilité des cordes avec densité singulière}, C. R. Acad. Sci. Paris Sér. I, 322 (1996), 365–370.


\bibitem{CaKoPa2014} R. A. Capistrano--Filho, A. F. Pazoto and V. Komornik, {\em Stabilization of the Gear-Grimshaw system on a periodic domain}, Commun. Contemp. Math. 16 (2014) 1--22.

\bibitem{capistranogallego2018} R. A. Capistrano--Filho and F. A. Gallego, {\em Asymptotic behavior of Boussinesq system of KdV–KdV type},
J. Differential Equations 265 (2018), 2341--2374.

\bibitem{capisgallegopazoto2016}
R. A. Capistrano--Filho, F. A. Gallego and A. F. Pazoto, {\em Neumann Boundary Controllability of the Gear Grimshaw System with Critical Size Restrictionson on the Spatial Domain},
Z. Angew. Math. Phys. 67 (5) (2016), 67:109.

\bibitem{capisgallegopazoto2017}
R. A. Capistrano--Filho, F. A. Gallego and A. F. Pazoto, {\em Boundary controllability of a nonlinear coupled system of two Korteweg–de Vries equations with critical
size restrictions on the spatial domain}, Math. Control Signals Systems 29 (1) (2017), Art. 6, 37 pp.

\bibitem{cerpapazoto2011}
E. Cerpa and A. F. Pazoto, {\em A note on the paper On the controllability of a coupled system
of two Korteweg-de Vries equations}, Commun. Contemp. Math 13 (2011), 183--189.

\bibitem{Dav3}
M. Dávila, {\em Estabilização de um sistema acoplado de equações tipo KdV}, Proceedings of the $45^\circ$ Seminário Brasileiro de Análise, Florianópolis, 1997, Vol. 1, 453--458.

\bibitem{DavCha2006} M. Dávila and F. S. Chaves,\emph{Infinite conservation laws for a system of Korteweg--de Vries type equations,}
Proceedings of DINCON 2006, Brazilian Conference on Dynamics, Control and Their Applications, May 22--26, 2006, Guaratinguetá, SP, Brazil (four pages).

\bibitem{DolRus1977} S. Dolecki and   D.L. Russell, {\em A general theory
of observation and control}, SIAM J. Control Opt. 15 (1977),
185--220.

\bibitem{GeaGri1984}
J. A. Gear and R. Grimshaw, {\em Weak and strong interactions between internal solitary waves,} Stud. Appl. Math. 70 (1984), 235--258.

\bibitem{Har1989}
A. Haraux,  {\em S\'eries lacunaires et contr\^ole semi-interne des vibrations d'une plaque rectangulaire},  J. Math. Pures Appl. 68 (1989), 457--465.

\bibitem{Har1990} A. Haraux, {\em Quelques méthodes et résultats récents en théorie de la contrôlabilité exacte}, Rapport de recherche No. 1317, INRIA Rocquencourt, Octobre 1990.

\bibitem{HarJaf991} A. Haraux and S. Jaffard, {\em Pointwise and spectral control of plate vibrations}, Rev. Mat. Iberoamericana 7 (1) (1991).

\bibitem{Har1994} A. Haraux, {\em Quelques propriétés des séries lacunaires utiles dans l’étude des systèmes élastiques}, Nonlinear partial differential equations and their applications. Collège de France Seminar,XII (Paris, 1991–1993 ), 113–124 Pitman Res. Notes Math.Ser., 302, Longman Sci.Tech.Harlow, 1994.

\bibitem{HirSat1981} R. Hirota and J. Satsuma, {\em Soliton solutions of a coupled Korteweg-de Vries equation}, Phys. Lett. A 85 (1981), 407--408.

\bibitem{HirSat1982} R. Hirota and J. Satsuma, {\em A Coupled KdV equation is one case of the four-reduction of the KP hierarchy}, J. Phys. Soc. Japan 51 (1982), 3390--3397.

\bibitem{Ing1936} A. E. Ingham, {\em Some trigonometrical inequalities with applications in the theory of series}, Math. Z. 41 (1936), 367--379.

\bibitem{JaTuZu1997} S. Jaffard, M. Tucsnak and E. Zuazua, {\em On a theorem of Ingham}, J. Fourier Anal. Appl., 3 (1997), 577–582.

\bibitem{Kom83} V. Komornik, {\em  Rapid boundary
stabilization of linear distributed systems},  SIAM J.
Control Optim. 35 (1997), 1591--1613.

\bibitem{KomLor2005} V. Komornik and P. Loreti, {\em Fourier Series in Control Theory},
Springer-Verlag, New York, 2005.

\bibitem{KomRusZha1991} V. Komornik, D. L. Russell and B.-Y. Zhang, {\em  Stabilisation de l'équation de Korteweg--de Vries}, C. R. Acad. Sci. Paris Sér. I Math. 312 (1991), 841--843.

\bibitem{Lions1988} J.-L. Lions,  {\em Contr\^olabilit\'e exacte et
stabilisation  de syst\`emes distribu\'es}, Volume I, Masson, Paris, 1988.

\bibitem{Lor2005} P. Loreti, {\em On some gap theorems}, Proceedings of the 11th Meeting of EWM, CWI Tract, 2005.

\bibitem{Meh2005}  M. Mehrenberger, {\em Critical length for a Beurling type theorem},
Bol. Un. Mat. Ital. B (8) 8 (2005), 251--258.

\bibitem{micuortega2000} S. Micu and J. H. Ortega, {\em On the controllability of a linear coupled system of Korteweg-de Vries equations}, in Mathematical and Numerical Aspects of Wave Propagation (Santiago de Compostela, 2000) (SIAM, Philadelphia, PA, 2000), 1020--1024.

\bibitem{micuortegapazoto2009}
S. Micu, J. H. Ortega and A. F. Pazoto, {\em On the Controllability of a Coupled system of two Korteweg-de Vries equation}, Commun. Contemp. Math.  11 (2009), 779--827.

\bibitem{PazSou} A. F. Pazoto and G. R. Souza, {\em Uniform stabilization of a nonlinear dispersive system}, Quart. Appl. Math. 72 (2014), 193--208.

\bibitem{rosier1997} L. Rosier, {\em Exact boundary controllability for the Korteweg–de Vries equation on a
bounded domain}, ESAIM Control Optim. Cal. Var. 2 (1997), 33--55.

\bibitem{Rus1978} D. L. Russell, {\em Controllability and stabilizability theory for linear partial differential equations. Recent progress and open questions}, SIAM Rev. 20 (1978), 639--739.
\end{thebibliography}
\end{document}